\documentclass[12pt,onecolumn]{IEEEtran}

\usepackage{setspace}
\RequirePackage{amsthm,amsmath,amsfonts,amssymb}

\usepackage{natbib}

\usepackage{float}

\usepackage[utf8]{inputenc}
\usepackage{bbm}
\usepackage{bm} 
\usepackage{dsfont}
\usepackage{paralist}
\usepackage{enumitem}
\usepackage[usenames]{color}
\usepackage{url}
\usepackage[colorlinks,
linkcolor=red,
anchorcolor=blue,
citecolor=blue
]{hyperref}
\usepackage{xspace,exscale,relsize}
\usepackage{fancybox,shadow}
\usepackage{graphicx}
\usepackage{caption}
\usepackage{subcaption} 
\usepackage{xcolor}
\usepackage{caption}
\usepackage{extarrows}
\usepackage{comment}
\usepackage{booktabs} 
\usepackage{makecell, multirow}

\usepackage{mathrsfs}
\makeatletter
\let\over=\@@over \let\overwithdelims=\@@overwithdelims
\let\atop=\@@atop \let\atopwithdelims=\@@atopwithdelims
\let\above=\@@above \let\abovewithdelims=\@@abovewithdelims
\makeatother
\usepackage{rotating}

	\usepackage{ifpdf}
	
	\usepackage{psfrag}
	
	\usepackage{prettyref}

	\usepackage{tikz}
	\usetikzlibrary{arrows}
	\tikzstyle{int}=[draw, fill=blue!20, minimum size=2em]
	\tikzstyle{dot}=[circle, draw, fill=blue!20, minimum size=2em]
	\tikzstyle{init} = [pin edge={to-,thin,black}]

	\usepackage[all]{xy}

	\usepackage{algorithm}
	\usepackage{algpseudocode}
	
	\usepackage{mathtools}

	\def\ba{{\boldsymbol a}}   \def\bA{\boldsymbol A}  
	\def\bb{{\boldsymbol b}}     
	     
	\def\bd{{\boldsymbol d}}     
	     \def\EE{\mathbb{E}}

	     \def\PP{\mathbb{P}}

	\def\bu{{\boldsymbol u}}     
	\def\bv{\boldsymbol v}     
	\def\bw{\boldsymbol w}   \def\bW{\boldsymbol W}  
	\def\bx{\boldsymbol x}     
	\def\by{\boldsymbol y}     
	\def\bz{\boldsymbol z}     
	
	\def\11{\mathbbm{1}}

	\def\calB{{\cal  B}}

	\def\calF{{\cal  F}} 
	\def\calG{{\cal  G}} 
	\def\calH{{\cal  H}} 
	 
	\def\calJ{{\cal  J}} 
	 
	\def\calL{{\cal  L}} 
	 
	\def\calN{{\cal  N}} 
	 
	\def\calP{{\cal  P}} \def\cP{{\cal  P}}
	\def\calQ{{\cal  Q}} 
	 
	\def\calS{{\cal  S}}

	\def\calV{{\cal  V}}

	\newcommand{\bfsym}[1]{\ensuremath{\boldsymbol{#1}}}

	           \def\bTheta {\bfsym {\Theta}}

	\def\bomega {\bfsym {\omega}}

	\def\bxi{\bfsym {\xi}}
	
	\DeclareMathOperator{\argmin}{argmin}

	\DeclareMathOperator{\diag}{diag}
	\DeclareMathOperator{\E}{E}

	\DeclareMathOperator{\sgn}{sgn}

	\def\E{{\mathbb{E}}}

	\def\Tr{{{\rm Tr}}}

	\graphicspath {{fig/}}
	
	\usepackage{ifthen}
	\newboolean{aos}
	\setboolean{aos}{FALSE}

	\ifx\eqref\undefined
	\newcommand{\eqref}[1]{~(\ref{#1})}
	\fi
	\ifx\mod\undefined
	\def\mod{\mathop{\rm mod}}
	\fi
	
	\usepackage{bm}

	\newcommand{\norm}[1]{{\left\Vert #1 \right\Vert_2}}

	\def\argmin{\mathop{\rm argmin}}

	\def\exp{\mathop{\rm exp}}

	\def\EE{\Expect}

	\def\PP{\mathbb{P}}

	\def\diag{\mathop{\rm diag}}

	\def\eqdef{\triangleq}

	\def\simiid{\stackrel{iid}{\sim}}
	
	\newcommand{\abs}[1]{\left| #1 \right|}

	\makeatletter
	\def\bbordermatrix#1{\begingroup \m@th
		\@tempdima 4.75\p@
		\setbox\z@\vbox{%
			\def\cr{\crcr\noalign{\kern2\p@\global\let\cr\endline}}%
			\ialign{$##$\hfil\kern2\p@\kern\@tempdima&\thinspace\hfil$##$\hfil
				&&\quad\hfil$##$\hfil\crcr
				\omit\strut\hfil\crcr\noalign{\kern-\baselineskip}%
				#1\crcr\omit\strut\cr}}%
		\setbox\tw@\vbox{\unvcopy\z@\global\setbox\@ne\lastbox}%
		\setbox\tw@\hbox{\unhbox\@ne\unskip\global\setbox\@ne\lastbox}%
		\setbox\tw@\hbox{$\kern\wd\@ne\kern-\@tempdima\left[\kern-\wd\@ne
			\global\setbox\@ne\vbox{\box\@ne\kern2\p@}%
			\vcenter{\kern-\ht\@ne\unvbox\z@\kern-\baselineskip}\,\right]$}%
		\null\;\vbox{\kern\ht\@ne\box\tw@}\endgroup}
	\makeatother
	
	\newcommand{\stepa}[1]{\overset{\rm (a)}{#1}}
	\newcommand{\stepb}[1]{\overset{\rm (b)}{#1}}
	\newcommand{\stepc}[1]{\overset{\rm (c)}{#1}}
	\newcommand{\stepd}[1]{\overset{\rm (d)}{#1}}

	\newcommand{\reals}{\mathbb{R}}
	\newcommand{\naturals}{\mathbb{N}}

	\newcommand{\Expect}{\mathbb{E}}
	
	\newcommand{\Prob}{\mathbb{P}}

	\newcommand{\iid}{iid\xspace}

	\newcommand{\pth}[1]{\left( #1 \right)}
	\newcommand{\qth}[1]{\left[ #1 \right]}
	\newcommand{\sth}[1]{\left\{ #1 \right\}}

	\definecolor{myblue}{rgb}{.8, .8, 1}
	\definecolor{mathblue}{rgb}{0.2472, 0.24, 0.6} 
	\definecolor{mathred}{rgb}{0.6, 0.24, 0.442893}
	\definecolor{mathyellow}{rgb}{0.6, 0.547014, 0.24}

	\newcommand{\tM}{{\tilde{M}}}

	\newrefformat{eq}{(\ref{#1})}
	\newrefformat{thm}{Theorem~\ref{#1}}
	\newrefformat{th}{Theorem~\ref{#1}}
	\newrefformat{chap}{Chapter~\ref{#1}}
	\newrefformat{sec}{Section~\ref{#1}}
	\newrefformat{seca}{Section~\ref{#1}}
	\newrefformat{algo}{Algorithm~\ref{#1}}
	\newrefformat{asmp}{Assumption~\ref{#1}}
	\newrefformat{fig}{Fig.~\ref{#1}}
	\newrefformat{tab}{Table~\ref{#1}}
	\newrefformat{rmk}{Remark~\ref{#1}}
	\newrefformat{clm}{Claim~\ref{#1}}
	\newrefformat{def}{Definition~\ref{#1}}
	\newrefformat{cor}{Corollary~\ref{#1}}
	\newrefformat{lmm}{Lemma~\ref{#1}}
	\newrefformat{prop}{Proposition~\ref{#1}}
	\newrefformat{pr}{Proposition~\ref{#1}}
	\newrefformat{property}{Property~\ref{#1}}
	\newrefformat{app}{Appendix~\ref{#1}}
	\newrefformat{apx}{Appendix~\ref{#1}}
	\newrefformat{ex}{Example~\ref{#1}}
	\newrefformat{exer}{Exercise~\ref{#1}}
	\newrefformat{soln}{Solution~\ref{#1}}
	
	\def\unifto{\mathop{{\mskip 3mu plus 2mu minus 1mu%
				\setbox0=\hbox{$\mathchar"3221$}%
				\raise.6ex\copy0\kern-\wd0%
				\lower0.5ex\hbox{$\mathchar"3221$}}\mskip 3mu plus 2mu minus 1mu}}

	\ifx\lesssim\undefined
	\def\simleq{{{\mskip 3mu plus 2mu minus 1mu%
				\setbox0=\hbox{$\mathchar"013C$}%
				\raise.2ex\copy0\kern-\wd0%
				\lower0.9ex\hbox{$\mathchar"0218$}}\mskip 3mu plus 2mu minus 1mu}}
	\else
	\def\simleq{\lesssim}
	\fi
	
	\ifx\gtrsim\undefined
	\def\simgeq{{{\mskip 3mu plus 2mu minus 1mu%
				\setbox0=\hbox{$\mathchar"013E$}%
				\raise.2ex\copy0\kern-\wd0%
				\lower0.9ex\hbox{$\mathchar"0218$}}\mskip 3mu plus 2mu minus 1mu}}
	\else
	\def\simgeq{\gtrsim}
	\fi

		\theoremstyle{definition}
		
		\newif\ifmapx
		{\catcode`/=0 \catcode`\\=12/gdef/mkillslash\#1{#1}}
		\edef\jobnametmp{\expandafter\string\csname ic_apx\endcsname}
		\edef\jobnameapx{\expandafter\mkillslash\jobnametmp}
		\edef\jobnameexpand{\jobname}
		\ifx\jobnameexpand\jobnameapx
		\mapxtrue
		\else
		\mapxfalse
		\fi

		\renewcommand{\hat}{\widehat}
		\renewcommand{\tilde}{\widetilde}

		\newcommand{\relu}{{\text{ReLU}~}}

		\newcommand{\dist}{{\sf dist}}

		\newtheorem{theorem}{Theorem}
		\newtheorem{lemma}{Lemma}

		\newtheorem{proposition}{Proposition}
		
		\newtheorem{definition}{Definition}
		
		\theoremstyle{definition}
		
		\newtheorem{remark}{Remark}
		
		\newtheorem{myassumption}{Assumption}
		
\newcommand{\phihat}{{\hat \phi}}
\newcommand{\phistar}{{\phi^*}}
\newcommand{\phitilde}{{\tilde \phi}}

\newcommand{\Thetahat}{{\hat \Theta}}
\newcommand{\Thetastar}{{\Theta^*}}
\newcommand{\Thetatilde}{{\tilde \Theta}}

\newcommand{\Sigmahat}{{\hat \Sigma}}
\newcommand{\Sigmastar}{{\Sigma^*}}
\newcommand{\Sigmatilde}{{\tilde \Sigma}}

\newcommand{\sparse}{{\sf SPN}}

\usepackage{physics} 

\begin{document}
	
	\ifpdf
	\DeclareGraphicsExtensions{.pgf,.jpg,.pdf}
	\graphicspath{{figures/}{plots/}}
	\fi
	\title{Minimax Theory of Likelihood-Based Deep Learning for Speckle Regression}
	
	
	\author{Soham Jana \thanks{S.~Jana is with the Department of Applied and Computational Mathematics and Statistics, University of Notre Dame, Notre Dame, IN, USA (correspondence to: \url{soham.jana@nd.edu}).}
		

	}

	\maketitle


	\begin{abstract}
Speckle noise is a multiplicative noise commonly encountered in coherent imaging modalities such as synthetic aperture radar, optical coherence tomography, and digital holography. Although deep learning methods, in practice, have achieved state-of-the-art performance for speckle denoising, their fundamental statistical limits remain largely unexplored. Unlike additive noise models, multiplicative speckle noise makes the regression function unidentifiable from the conditional mean, rendering conventional least-squares-based deep learning approaches inapplicable.

We study the minimax estimation of smooth nonparametric regression functions using likelihood-based deep neural network (DNN) estimators under a model with both multiplicative speckle noise and additive Gaussian noise. Our framework accommodates both low-dimensional and sparse high-dimensional features. We establish finite-sample upper bounds on the estimation error of the proposed DNN estimators and derive minimax lower bounds for nonparametric function recovery under our model, showing that they match up to logarithmic factors in the sample size. Moreover, these minimax rates coincide, up to logarithmic factors, with those for nonparametric regression under additive Gaussian noise alone, demonstrating that the intrinsic difficulty of estimation remains essentially unchanged despite the challenges posed by multiplicative speckle noise. Numerical experiments further supports consistency of our DNN-based despeckling methods and demonstrate their effectiveness.

	\end{abstract}
	
	{\it Keywords: Speckle Noise, Nonparametric regression, Deep neural networks, Sparsity}

	\section{Introduction}
	
	\subsection{Motivation and problem formulation}
	
	Signal recovery in the presence of speckle (multiplicative) noise is a fundamental problem in statistical imaging and machine learning. In coherent imaging systems, the recorded measurements are inherently contaminated by speckle, a granular interference phenomenon arising from the coherent superposition of scattered optical waves \cite{Goodman2007}. Speckle acts as multiplicative noise that degrades image contrast, obscures fine structural details, and fundamentally limits the accuracy of image reconstruction and downstream quantitative analysis. Consequently, developing statistically efficient estimation procedures under multiplicative speckle noise has become a central challenge in modern computational imaging, including Synthetic Aperture Radar (SAR) \cite{Lopez-MartinezFabregas}, Optical Coherence Tomography (OCT) \cite{SXK}, ultrasound imaging \cite{achim2001novel}, digital holography \cite{dh}, and Quantitative Phase Imaging (QPI) \cite{kemper2007digital}. 
	Many such imaging modalities naturally give rise to nonparametric regression problems with multidimensional or high-dimensional covariates. In two-dimensional SAR and medical ultrasound, the covariates correspond to spatial pixel locations, while the underlying regression function describes the coherent reflectivity of the observed object \cite{goodman1976speckle,oliver2004understanding}. Diffusion-weighted MRI (dMRI) models signals jointly over spatial location, diffusion-gradient direction, and acquisition parameters \cite{tournier2011diffusion}, resulting in substantially higher-dimensional regression problems. Many modern imaging modalities acquire observations indexed not only by spatial location but also by acquisition parameters, spectral channels, temporal measurements, or molecular variables, naturally leading to regression problems with high-dimensional covariates \cite{park2018quantitative,lu2014medical}. Despite their high ambient dimension, these signals often possess substantial low-dimensional or sparse structure. Consequently, sparsity-based reconstruction methods have become widely used in coherent imaging, for example, see \cite{gao2023iterative} for $\ell_1$- and $\ell_0$-regularized methods employed in the context of QPI and digital holography to recover phase information from highly undersampled measurements.

	Motivated by these applications, we consider the following nonparametric regression model with multiplicative and additive noise:
	\begin{align}
		\label{eq:model}
		y_i=f^*(\bx_i)\xi_i+\tau_i,
		\qquad
		i=1,\ldots,n,
	\end{align}
	where $
	f^*(\bx):\reals^d\mapsto\reals,
	$ is an unknown, possibly nonlinear, regression function generating the corresponding clean image measurement (or phase value), $\bx_i\in\mathbb{R}^d$ denotes the observed feature vector, with dimension $d\geq 1$, $y_i\in\mathbb{R}$ denotes the observed measurement (e.g., a pixel or phase value), $\xi_i\simiid N(0,1)$ models the {\it speckle noise}, and $\tau_i\simiid N(0,\sigma_\tau^2)$ represents additive measurement noise. The Gaussian assumption on the additive noise is standard in statistical learning. For the multiplicative component, we adopt the fully developed speckle model \cite{argenti2013tutorial}, in which coherent wavefronts scattered by numerous microscopic surface variations give rise to a Gaussian approximation through the central limit theorem \cite{goodman1975statistical}. As noted in \cite{malekian2025speckle}, assuming speckle standard deviation to be unity incurs no loss of generality, since under the assumption of a known speckle standard deviation, the observations can be rescaled to normalize the variance of $\xi_i$.

	In the current work, we investigate minimax guarantees for estimating $f^*$
	under the speckle noise model \eqref{eq:model} using a deep neural network (DNN). DNNs have become a standard approach for image denoising applications by learning a direct mapping from noisy observations to clean images. For example, \cite{tang2019sar} proposed a patch-based DNN for SAR image despeckling and later demonstrated that DNN-based despeckling improves downstream land-cover classification \cite{tang2018land}. While more sophisticated deep architectures, particularly convolutional neural networks, now dominate practical despeckling methods due to their strong empirical performance \cite{chierchia2017sar,dalsasso2020sar}, fully connected DNNs provide a natural starting point for theoretical analysis, which motivates our study. To this end, we emphasize that the multiplicative speckle noise model \eqref{eq:model} requires fundamentally different analytical techniques from those developed for the classical additive noise setting. Unlike additive noise setup, the regression function $f^*$ in \eqref{eq:model} is not identifiable from the conditional mean, as we have
	\begin{equation*}
		\E[y_i \mid \bx_i] = f^*(\bx_i)\,\E[\xi_i] = 0.
	\end{equation*}
	Consequently, the classical strategy of constructing DNN estimators by minimizing the squared-error loss, which underlies most existing statistical analyses of DNN estimators \cite{bauer2019deep,kohler2021rate}, is no longer applicable. Moreover, the fundamental statistical limits of nonparametric estimation under multiplicative speckle noise remain largely unexplored. Existing theoretical works primarily analyze local polynomial or wavelet-based estimators. In particular, matching minimax upper and lower bounds with a general value of the additive noise standard deviation $\sigma_\tau$ have only been established in the one-dimensional setting $d=1$ by \cite{malekian2025speckle}, while \cite{chesneau2020nonparametric,benhaddou2022estimation} derive only upper bounds for fixed-dimensional settings. Developing statistical guarantees for flexible function classes such as deep neural networks is therefore of considerable interest. Unlike local polynomial or wavelet-based methods, deep neural networks provide a model-agnostic framework for approximating complex nonlinear functions without requiring carefully designed local approximations or basis expansions. In view of these challenges, we ask the following questions.
	\begin{itemize}
		\item How can we construct a DNN estimator to efficiently estimate $f^*$?
		\item What risk bounds hold for deep learning estimators, and are they minimax optimal? This extends the proposed problem of \cite{malekian2025speckle} to the multivariate setup, aiming to explore whether speckle denoising is equally challenging to mitigate compared to the purely additive noise model.
		\item How can we perform signal recovery with high-dimensional features when the underlying signal depends on only a few unknown sparse coordinates and speckle noise is present?
	\end{itemize}
	We answer these questions in the regime where $\sigma_\tau$ is of constant order. Specifically, we propose a likelihood-based deep neural network estimator obtained by minimizing the negative log-likelihood
	\begin{align}
		\label{eq:likelihood-full}
		\ell_{\text{FULL}}(f)
		\eqdef
		\ell_{\text{FULL}}\pth{f \mid \sth{(y_i,\bx_i)}_{i=1}^n}
		=
		\frac{1}{n}\sum_{i=1}^n
		\frac{y_i^2}{\{f(\bx_i)\}^2+\sigma_\tau^2}
		+
		\frac{1}{n}\sum_{i=1}^n
		\log\!\bigl(\{f(\bx_i)\}^2+\sigma_\tau^2\bigr)
	\end{align}
	over an appropriately constrained class of deep neural networks. For the sparse high-dimensional setting, we develop an $\ell_1$-penalized likelihood estimator that simultaneously performs variable selection and function estimation. We establish theoretical guarantees for these estimators, including near-minimax optimality, as described below.
	
	\subsection{Our contributions}
	
	To the best of our knowledge, this is the first work to provide a minimax theoretical foundation for deep neural network-based despeckling of nonparametric signals. Assuming the true $f^*$ is a smooth nonparametric function, we establish finite-sample risk bounds for the proposed likelihood-based estimators and prove their near-minimax optimality by deriving matching minimax lower bounds under the speckle noise model, up to logarithmic factors. Furthermore, we show that the resulting minimax rates coincide, up to logarithmic factors, with those for nonparametric regression under additive Gaussian noise alone, demonstrating that multiplicative speckle noise does not fundamentally increase the statistical difficulty of function estimation, even though it significantly alters the estimation strategy. Numerical experiments corroborate the theoretical predictions and illustrate the practical effectiveness of the proposed methods. The main contributions are summarized below.
	
	\begin{enumerate}
		\item \textbf{Theoretical analysis of deep learning estimator without sample splitting:} Our work first presents a deep neural network estimator for speckle denoising when the true generating functional is a smooth function of low-dimensional feature vectors.
		Our estimation strategy optimizes the negative log-likelihood presented in \eqref{eq:likelihood-full}. In particular, when the target function $f^*$ belongs to a class of smooth functions with degree adjusted smoothness $\gamma^*$ (see \prettyref{def:beta-C-smooth} for the details), our estimator achieves a mean squared error rate of $\tilde O(n^{-{2\gamma^*\over 2\gamma^*+1}})$, where $\tilde O$ hides logarithmic factors in the sample size $n$. Note that the above rate is similar to the classical nonparametric regression rate $O(n^{-{2\beta\over 2\beta +d}})$ for estimating $\beta$-smooth functions in $d$-dimensions with additive noise \cite{gyorfi2002distribution} (here $\gamma^*$ is similar to $\beta/d$, see \prettyref{rmk:smoothness} for a comparison). Our proof technique presents an approach based on the empirical process theory and covering bound construction, that establishes the desired error guarantees without requiring any sample splitting.
		
		\item \textbf{Guarantees for sparse estimation:} In the presence of high-dimensional feature vectors, assuming that the data generating function is smooth and depends on unknown sparse locations of the feature vector, we present a penalized likelihood based deep-learning strategy for functional estimation. Our estimator achieves a mean squared error rate of $\tilde O(n^{-{2\gamma^*\over 2\gamma^*+1}})$, where $\tilde O$ hides logarithmic factors in the sample size $n$ and the target function $f^*$ belongs to a class of smooth functions with degree adjusted smoothness $\gamma^*$. Our results hold even when $d$ is significantly larger, provided the number of locations affecting the construction of $f^*$ is small.
		
		\item \textbf{Minimax lower bound:} We provide a minimax lower bound of $\Omega_{\sigma_\tau}(n^{-{2\gamma^*\over 2\gamma^*+1}})$ on the mean squared error of estimation, when the regression function $f^*$ belongs to a class of functions with degree adjusted smoothness $\gamma^*$ and $\Omega_{\sigma_\tau}$ hides multiplicative constants depending on $\sigma_{\tau}$. Our results show that the deep learning estimators are optimal up to logarithmic factors in the sample size. Our result extends the lower bound of \cite{malekian2025speckle} in the case $d=1$ to the multidimensional setting.
		
		\item \textbf{Empirical validation:} We present numerical results to validate the theoretical guarantees of our deep neural network estimators. In the low-dimensional setting, our experiments investigate the estimation accuracy of the proposed estimator under varying sample sizes, network depths, and numbers of training epochs. In the high-dimensional setting, we demonstrate that our estimator successfully identifies the active features and that the detection accuracy improves as the signal strength increases.
		
	\end{enumerate}

	\subsection{Related literature}
	
	Classical nonparametric regression has been extensively studied over the past several decades. \cite{fan1996local} showed that local polynomial estimators attain the multidimensional minimax rates established by \cite{stone1982global}. Similar optimal error guarantees were obtained for kernel regression estimators by \cite{wand1994kernel}. Wavelet thresholding provides a spatially adaptive alternative through sparse representations. In particular, \cite{donoho1995wavelet,donoho1998minimax} proved that soft-thresholding estimators achieve near-minimax rates over Besov spaces simultaneously for a broad range of smoothness and integrability parameters. While these methods enjoy strong theoretical guarantees, they exhibit important practical limitations. Local polynomial estimators become increasingly ineffective in moderate to high dimensions, whereas wavelet-based methods require carefully selecting a basis adapted to the underlying signal and application.
	
	The theory of nonparametric regression under multiplicative noise is comparatively less developed. The seminal work of \cite{korostelev2012minimax} provides a comprehensive analysis of boundary estimation under multiplicative noise with finite-valued error distributions and characterizes minimax rates using local polynomial estimators. For Gaussian speckle noise, \cite{malekian2025speckle} established the minimax rate
	$\pth{\sigma_\tau^2\over n}^{-2\beta/ (2\beta+1)}$ in the one-dimensional setting $d=1$, where $\beta$ denotes the smoothness of the regression function $f^*$. The optimal rate is achieved by a local polynomial estimator. Earlier multidimensional results \cite{chesneau2020nonparametric,benhaddou2022estimation} considered the regime $d=O(1)$ and $\sigma_\tau$ is of constant order, and established the upper bound
	$
	O\left(n^{-2\beta/(2\beta+d)}\right)$
	using wavelet and Laguerre-polynomial-based estimators, respectively. However, these works leave open the question of whether this rate is minimax optimal in the multidimensional setting.
	
	Deep neural networks provide a flexible alternative to classical nonparametric estimators by learning complex function classes without relying on handcrafted local approximations or basis expansions. Their empirical success for speckle denoising has been widely demonstrated; see, for example, \cite{mohan2021deep,hyun2019beamforming}. Nevertheless, theoretical understanding of deep learning methods for multiplicative-noise models remains limited. A recent line of work \cite{chen2024novel,chen2026maximum} investigated likelihood-based deep learning approaches for despeckling. Subsequently, \cite{chen2025multilook,xing2025minimax} established statistical guarantees for likelihood-based estimation when the underlying signal belongs to structured function classes with complexity comparable to locally polynomial models and the optimization is performed over parameterizations closely matched to the true signal. In contrast, the present work studies a substantially richer nonparametric function class and analyzes likelihood-based estimation over deep neural network models in a fully model-agnostic framework.
	
	There is also a growing body of literature on deep neural networks for nonparametric regression under additive noise models \cite{yarotsky2017error,kohler2021rate,bauer2019deep,schmidt2020nonparametric,bhattacharya2024deep}. Earlier empirical studies \cite{lecun2015deep,mousavi2015deep} demonstrated the ability of deep neural networks to exploit low-dimensional structures in high-dimensional data. Our analysis builds upon empirical process techniques to uniformly control the stochastic fluctuations of the likelihood over the neural network hypothesis class. Similar empirical process arguments have recently been employed for additive-noise models in both low-dimensional and sparse estimation settings \cite{fan2024noise,fan2025factor}. The key distinction in the present work is that we analyze a likelihood corresponding to multiplicative speckle noise rather than the conventional squared-error loss. Although neural-network-based methods for multiplicative-noise denoising have also been proposed \cite{zhao2021dual,moradkhani2022deep}, these works do not provide statistical guarantees for nonparametric function estimation under the multiplicative noise model considered here.

	\subsection{Notations}
	
	For two sequences of numbers $\{a_n\}$ and $\{b_n\}$, indexed with $n$, we say $a_n=O(b_n)$, or equivalently $b_n=\Omega(a_n)$, if there exist constants $C>0$, and $M>0$, such that for all $n>M$, $|a_n|\le C |b_n|$. Also, $a_n=o(b_n)$, or equivalently $b_n=\omega(a_n)$ if $\lim_{n\to\infty} {a_n}/{b_n} =0$. We use bold letters $\bu,\bv, \bx,\by$ to denote vectors and regular letters to denote scalars, unless mentioned otherwise. Given two numbers $x,y$, denote $x\vee y =\max\{x,y\}$ and $x\wedge y=\min\{x,y\}$. Matrices are usually denoted using capital letters, such as $B, \Theta,\Sigma$, etc. We write $a_n = \Theta (b_n)$ if $a_n = O(b_n)$ and $b_n = O(a_n)$. 
	The \emph{Hamming distance} between two vectors $\bu,\bv \in \{0,1\}^m$, is denoted with $\rho (\bu,\bv)$, and is defined as  
	$
	\rho(\bu,\bv) := \sum_{i=1}^m \mathbbm{1}\{u_i \neq v_i\},
	$
	where $\mathbbm{1}\{\cdot\}$ denotes the indicator function. Given a matrix $\bA=[A_{i,j}]_{i\in [n], j\in [m]}$, we define the norms $\|\bA\| = \sup_{\bx \in \mathbb{R}^m, \|\bx\|_2=1} \| \bA \bx\|_2$, $\|\bA\|_{HS} = \sqrt{\sum_{i,j} A_{i,j}^2}$, and $\|\bA\|_{\max}=\max_{i\in [n],j\in [m]} |A_{i,j}|$. Moreover, we use $\lambda_{\min}(\bA)$ and $\nu_{\min}(\bA)$ to denote its minimum eigenvalue and singular value, respectively.
	For function $\phi: \mathbb{R}^d \rightarrow \mathbb{R}$, we use the following notations:
	$$\lVert \phi \rVert^2_2 :=\int_{(-\infty,\infty)^d} \phi(\bx)^2 d\bx, \quad
	\lVert \phi \rVert^2_{2,\calQ} :=\int \phi(\bx)^2 d\calQ(\bx),
	\quad
	\phi_{\max} :=\lVert \phi \rVert_{\infty}:= \sup_{\bx} |\phi(\bx)|.$$
	For a vector $\bx$, we denote its Euclidean norm by $\|\bx\|_2$. We also use $\ell_0,\ell_1$ to denote the norm indicating the number of nonzero coordinates and the norm of absolute sum, respectively.
	Given non-random points $\bx_1,\dots,\bx_n\in \reals^d$, define the sample-based squared norm 
	$
	\|\phi\|_n^2=\frac 1n \sum_{i=1}^n (\phi(\bx_i))^2.
	$
	Given a random variable $\xi$, define its subgaussian norm given by $\|\xi\|_{\rm subgau} = \inf\{t>0 : \mathbb{E} (\exp (\xi^2/t^2)) \leq 2\}$ and its subgexponential norm given by $\|\xi\|_{\rm subexp} = \inf\{t>0 : \mathbb{E} (\exp (\xi/t)) \leq 2\}$.
	Denote by $\calN(\delta,\calV,\norm{\cdot})$ the $\delta$-covering number of a set $\calV$ with respect to a norm $\norm{\cdot}$.
	\subsection{Organization of our paper}
	
	The paper is organized as follows. We present major definitions in \prettyref{sec:prep}. We state our main theoretical results regarding the despeckling problem in \prettyref{sec:assumptions-main-results}. We present our numerical results in \prettyref{sec:simulation}. \prettyref{sec:likelihood-diff} presents a short sketch our our proof strategy. The remaining sections provide the technical details of the proof and more relevant results.

	\section{Theoretical contributions}
	\label{sec:theory}
	
	\subsection{Relevant definitions for the class of target functions and deep neural network estimators}
	\label{sec:prep}
	We present a few definitions relevant to the neural network construction and the smoothness of the target function class that we use throughout the paper.
	We build our model using a fully connected deep neural network with ReLU activation $\bar \sigma(\cdot)=\max\sth{\cdot,0}$. Before presenting our methodology and results, we provide definitions that we will rely on throughout the manuscript.
	
	\begin{definition}[Deep \relu network class]
		\label{def:relu-class}
		Let $L$ be any positive integer and $\bd = (d_1, . . . , d_{L+1}) \in \naturals^{L+1}$. A deep ReLU network $g:\reals^{d_0} \to \reals^{d_{L+1}}$ is given as in the form
		\begin{align}
			\label{eq:relu}
			g(\bx)
			=
			\calL_{L+1} \circ \bar\sigma \circ \calL_{L} \circ \bar\sigma \circ \cdots \circ \calL_2 \circ \bar\sigma \circ \calL_1(\bx),
		\end{align}
		where $\calL_{\ell}(\bz) = \bW_{\ell}\bz + \bb_\ell$ is a linear transformation with the weight parameters $\bW_{\ell} \in \reals^{d_{\ell}\times d_{\ell-1}},\bb_\ell \in \reals^{d_{\ell}}$, and $\bar\sigma : \reals^{d_{\ell}}\mapsto \reals^{d_\ell}$ applies the ReLU activation function coordinatewise.
		For any $L\in \naturals,\bd\in \naturals^{L+1}, B,M\in \reals^+\cup \{\infty\}$, the deep \relu network family $\calG(L,\bd,M,B)$ with truncation level $M$, depth $L$, width vector $\bd$, and weight bound $B$ is given as
		\begin{align*}
			\calG(L,\bd,M,B)
			=
			\sth{\Tr_M(g(\bx))\vee M^{-1} : g \text{ of form \eqref{eq:relu} with } \|\bW_{\ell}\|_{\max}\leq B,\|\bb_{\ell}\|_{\max}\leq B},
		\end{align*}
		where $\Tr_{M}(\cdot)$ is the coordinatewise truncation operator given by $[\Tr_M(\bz)]_i=\sgn(z_i)(|z_i|\wedge M)$ and $\norm{\cdot}_{\max}$ denotes the supremum norm of a vector. The class of deep \relu networks with depth $L$ and width $N$ is given by the specific case $\bd=(d_{in},N,N,\dots,N,d_{out})$, and we denote it throughout the text by $\calG(L,d_{in},d_{out}, N,M,B)$. Given two neural net parameters $\phi=\sth{(\bW_\ell,\bb_\ell)}_{\ell=1}^{L+1},\breve\phi=\{(\breve\bW_\ell,\breve\bb_\ell)\}_{\ell=1}^{L+1}$ from $\calG(L,d,1,N,M,B)$, define the max-norm between $\phi,\breve\phi$ as 
		\begin{align}
			\label{eq:metric-neural-net}
			\|\phi - \breve \phi\|_\infty = \max_{1\le \ell\le L+1} \left( \|\bb_\ell - \breve{\bb}_{\ell}\|_\infty \lor \|\bW_\ell - \breve{\bW}_\ell \|_{\max}\right).
		\end{align}
	\end{definition}
	
	We will use the following class of hierarchical composition functions to model $f^*$.
	
	\begin{definition}[$(\beta,C)$-smooth functions]
		\label{def:beta-C-smooth}
		A $d$-variate function $f$ is called $(\beta,C)$-smooth for $\beta,C>0$ if the following is satisfied. Decompose $\beta$ into integer part $r\geq 0$ and fraction part $0<s<1$. Then, for every non-negative sequence $\alpha \in \naturals^d$ with $\sum_{j=1}^d \alpha_j = r$, the partial derivative $(\partial f)/(\partial x_1^{\alpha_1}\dots x_d^{\alpha_d})$ exists, and
		\[
		\left|{\partial^r f\over \partial x_1^{\alpha_1}\dots \partial x_d^{\alpha_d}}(\bx)
		-
		{\partial^rf\over \partial x_1^{\alpha_1}\dots \partial x_d^{\alpha_d}}(\bz)\right|
		\leq C \|\bx - \bz\|^s_2.
		\]
	\end{definition}
	
	\begin{definition}[Hierarchical composition of smooth functions \cite{kohler2021rate,fan2024factor}]
		\label{def:hierarchical}
		Fix a constant $C>0$. Let $\calH(d,l,\calP)$ denote the class of $l$-depth and $d$-variate hierarchical composition of $(\beta,C)$-smooth functions for $(\beta,t)$ in a set $\calP$ with
		\[
		\label{eq:hierarchical}
		\calP \subset [1,\infty)\times \naturals^+, \qquad \sup_{(\beta,t)\in \calP} (\beta \vee t) <\infty
		\]
		\begin{itemize}
			\item $(l=1)$ We have the set of all $t$-variate functions with $(\beta,C)$ smoothness
			\begin{align*}
				\calH(d, 1,\calP)
				=
				&\left\{h : \reals^d \mapsto \reals : h(\bx) = g(\bx_{\calJ}), \text{ where } g : \reals^t \mapsto \reals \text{ is }\right. \\
				&\left.\text{$(\beta,C)$-smooth for some $(\beta, t) \in \calP$ and $\calJ\subseteq [d],|\calJ|=t$}\right\}
			\end{align*}
			\item $(l\geq 2)$ We recursively define $\calH(d,l,\calP)$ as
			\begin{align*}
				\calH(d, l,\calP)
				=
				&\left\{h : \reals^d \mapsto \reals : h(\bx) = g(f_1(\bx),\dots, f_t(\bx)), \text{ where } g : \reals^t \mapsto \reals \text{ is }  \right. \\
				&\left.\text{$(\beta,C)$-smooth for some $(\beta, t) \in \calP$ and $f_i\in \calH(d,l-1,\calP),i\in [t]$}\right\}
			\end{align*}
		\end{itemize}
	\end{definition}
	
	Basically, $\calH(d,l,\calP)$ consists of the $l$-fold compositions of $t$-variate functions of $(\beta, C)$ smoothness for any $(t, \beta) \in \cP$.
	The accuracy of estimating $f^*\in \calH(d, l,\calP)$ will be quantified by the parameter $\gamma^*$ indicating the hardness of the above composition class.
	
	\begin{definition}[Dimension adjusted smoothness parameter of $\calH(d, l,\calP)$]
		\label{def:dim-adjusted-smoothness}
		Given any function class $\calP$ satisfying \eqref{def:hierarchical} the hardness quantifier $\gamma^*$ of the worst case error of approximating any function in $\calH(d, l,\calP)$ by a deep \relu network is quantified by 
		\begin{align}\label{eq:gamma*}
			\gamma^* ={\beta^* \over d^*} \text{ with } (\beta^*, d^*) = \argmin_{(\beta,t)\in \calP}{\beta\over t}.
		\end{align}
		In view of \cite{kohler2021rate}, we restrict to the case where all the compositions in $\calH(d,l,\calP)$ have a smoothness parameter $\beta \geq 1$ to simplify the presentation. 
	\end{definition}

	\subsection{Error upper bounds for likelihood based DNN estimators}
	
	\label{sec:assumptions-main-results}
	Note that the full data likelihood \eqref{eq:likelihood-full} is equivalent to the likelihood corresponding to the model $y_i=\tilde f(\bx_i)\tilde \xi_i$ where $ (\tilde f(\bx_i))^2=( f^*(\bx_i))^2+\sigma_\tau^2$ and $\tilde \xi_i\simiid N(0,1)$. In view of the above, it is information theoretically impossible to separate the estimation of $f^*(\cdot)$ from the estimation of $(f^*(\cdot))^2+\sigma_\tau^2$, without making further structural assumptions on $f^*$ and $\sigma_{\tau}$. In view of the above, for the current work, we assume $\sigma_\tau$ is known and primarily focus on recovering the signal $f^*$. Similar restrictions were also considered in previous works, such as \cite{malekian2025speckle}.

	\begin{myassumption}
		\label{asmp:model-and-hyperparameters}
		
		We make the following assumptions on the data generating function and parameters, and the deep neural networks used to construct the estimators.
		\begin{enumerate}[label=(\alph*)]
			\item \label{pt:true-functions} The true data generating function $f^*$ satisfies
			$f^*\in \calH(d,l,\calP)$, and $\calP$ has the dimension-adjusted smoothness $\gamma^*>\frac 12$, where $\gamma^*$ is given by \prettyref{def:dim-adjusted-smoothness}. The coordinates of $\bx$ reside in $[-K,K]$ for some constant $K>0$. The data generating function $f^*$ satisfies $f_{\min}\leq f^{*}(\bx)\leq f_{\max}$ for all $\bx$ in the domain, where $0<M^{-1}<f_{\min}< f_{\max}<M$ are constants. The variance parameter $\sigma_\tau$ is assumed to be known and it is at most a constant. The feature vectors $(\bx_1,\dots,\bx_n)$ are fixed.
			
			\item \label{pt:hyperparam} We choose $M>0$ as a large constant. The estimators are chosen from the class of deep ReLU network $\calG(L,d,1,N,M,B)$ with the following hyperparameters. For constants $c_1,\dots,c_6$, which only depend on $l$ and $\mathcal{P}$ of $\mathcal{H}(d, l,\mathcal{P})$, we have
			\begin{align*}
				\begin{gathered}
					c_1 \le L \leq c_2,\quad  c_3 \log n \le \log B \leq c_4\log n,
					\quad
					c_5 (n/\log n)^{\frac{1}{4\gamma^*+2}} \le N \leq c_6 (n/\log n)^{\frac{1}{4\gamma^*+2}}.
				\end{gathered}
			\end{align*}
		\end{enumerate}
	\end{myassumption}
	\begin{remark}[Explaining the assumptions]
		\prettyref{asmp:model-and-hyperparameters}\ref{pt:true-functions} imposes regularity conditions on the data-generating process. We assume that the unknown regression function $f^*$ belongs to the the smooth, hierarchical, function class $\mathcal{H}(d,l,\mathcal{P})$, that can be efficiently represented by deep neural networks. The dimension-adjusted smoothness condition $\gamma^*>1/2$ ensures sufficient regularity for nonparametric estimation. Since the observations are corrupted by multiplicative speckle noise, we further assume that $f^*$ is uniformly bounded away from zero and infinity. The lower bounds on both the true and model functions ensures that the likelihood objective \eqref{eq:likelihood-full} arising the multiplicative noise model is nontrivial, and the upper bounds on the relevant feature coordinates ensures that the space is compact.
		
		\prettyref{asmp:model-and-hyperparameters}\ref{pt:hyperparam} specifies the network architecture and its scaling with the sample size. The network depth is chosen to be a constant, while the common width $N\asymp (n/\log n)^{1/(4\gamma^*+2)}$ is selected for technical reasons and helps us to ensure that the approximation error of the network and estimator variance are balanced. The bound on the network weights is allowed to grow polynomially with the sample size in a way that helps us control the complexity of the hypothesis class, while allowing sufficient strength for estimating the target functional. These choices are standard in the minimax analysis of deep ReLU networks and yield the optimal convergence rates established in our results, for example see \cite{kohler2021rate,fan2024noise,fan2025factor}.
	\end{remark}
	Given any function $f$, consider fixed-feature $(\bx_1,\dots,\bx_n)$-based squared norm $\|f\|_n^2=\frac 1n\sum_{i=1}^n f^2(\bx_i)$. In view of the above assumptions, we have the following estimation guarantees.
	\begin{theorem}[Error upper bound with finite dimensional features]
		\label{thm:finite-dim}
		Assume that $f^*$ satisfies  \prettyref{asmp:model-and-hyperparameters}\ref{pt:true-functions} and $d$ is a constant.
		Consider the Deep-NN estimator $f_\phihat(\bx)$ with weights $\phihat$ that satisfies the hyperparameter constraints in \prettyref{asmp:model-and-hyperparameters}\ref{pt:hyperparam} and minimizes the full data likelihood given in \eqref{eq:likelihood-full}
		$$
		f_\phihat 
		\in \argmin_{f_\phi\in \calG(L,d,1,N,M,B)}
		\ell_{\text {FULL}}\pth{f_\phi|\sth{(y_i,\bx_i)}_{i=1}^n}
		$$ 
		Then, there are constants $c_4,c_5>0$ such that, with probability at least $1-{1\over n^{c_5}}$
		\begin{align*}
			\|f_\phihat-f^*\|_n^2 \le c_4\left({\log n}\right)^{\frac{2\gamma^*}{2\gamma^*+1}}
			\cdot n^{-\frac{2\gamma^*}{2\gamma^*+1}}.
		\end{align*} 
	\end{theorem}
	
	Next, we present our estimation guarantees for the high-dimensional feature setup.
	
	\begin{theorem}[Error upper bound in sparse high-dimensional settings]
		\label{thm:sparse}
		Assume that the data generating function $f^*(\bx)$ only depends on $\bx$ via a finite set of unknown coordinates $\bx_\calJ$, where the size of the coordinate set satisfies $|\calJ|<<d$. Let $f^*$ satisfy \prettyref{asmp:model-and-hyperparameters}\ref{pt:true-functions} and the feature dimension $d$ satisfies $d\leq n^{c_0}$ for some constant $c_0>0$.
		Consider the Deep-NN estimator $f_\phihat(\bx)$  constructed by maximizing the $\ell_1$-penalized data likelihood given in \eqref{eq:likelihood-full}
		\begin{align}
			\label{eq:sparse-loss}
			f_\phihat,\Thetahat \in \argmin_{f_\phi\in \calG(L,d,1,N,M,B)\atop \Theta\in \reals^{d\times N}}
			\ell_{\text {FULL}}\pth{f_\phi|\sth{(y_i,\Tr_M\{\Theta^\top \bx_i\})}_{i=1}^n}
			+\lambda\sum_{i,j}\psi_\tau(\Theta_{i,j}),
		\end{align} 
		with hyperparameters as in \prettyref{asmp:model-and-hyperparameters}\ref{pt:hyperparam}, $\psi_\tau(x)=\frac{|x|}{\tau} \wedge 1$ and the tuning parameters $\lambda,\tau$ are given as
		\begin{align*}
			&\lambda = c_2 {\log(ndN) + L\log(BN)\over n},\quad 
			\tau^{-1} \geq  c_3(BN)^{L+1}Nd{n^3},
		\end{align*}
		for constant $c_2,c_3>0$. Denote the new estimator $\hat f^\sparse(\bx)$ as
		$
		\hat f^\sparse(\bx)
		=\hat f^\sparse(\bx;\hat \phi,\hat\Theta)
		=f_{\phihat}(\Tr_M(\hat\Theta^\top \bx)).
		$
		Then there are constants $c_4,c_5>0$ such that, with probability at least $1-{1\over n^{c_5}}$, $\hat f^\sparse$ satisfies
		\begin{align*}
			\|\hat f^\sparse-f^*\|_n^2 \le c_4\left({\log n}\right)^{\frac{2\gamma^*}{2\gamma^*+1}}
			\cdot n^{-\frac{2\gamma^*}{2\gamma^*+1}}.
		\end{align*} 
	\end{theorem}
	
	\begin{remark}[Population based loss functions for random features]
		
		When the feature vectors $(\bx_1,\dots,\bx_n)$ are sampled from some distribution $\calQ$, it is often of interest to obtain error bounds with respect to the population norm $\norm{\cdot}_{2,\calQ}$ given by
		\begin{align*}
			\|f\|_{2,\calQ}^2
			=\EE[f^2(\bx)]=\int f^2(\bx) d\calQ(\bx).
		\end{align*}
		Then one can convert the error guarantee for the fixed-feature-based norm $\norm{\cdot}_n$ to guarantees on $\|\cdot \|_2$. For the purpose of demonstration, consider the setup of a constant $d$. Then we have the following result.
		\begin{proposition}
			\label{prop:population-empirical-connection}
			Suppose that the feature vectors $\bx_1,\dots,\bx_n$ are drawn \iid from a distribution $\calQ$ on $\reals^d$ and the degree-adjusted smoothness parameter $\gamma^*$ of the data generating function class satisfies $\gamma^*>\frac 12$. Then the neural network $f_{\phihat}$ given in \prettyref{thm:finite-dim} satisfies that, for some constants $c,C>0$,
			$$
			\PP\qth{\|f_\phihat-f^*\|_{2,\calQ}^2\leq 
				C\pth{\|f_\phihat-f^*\|_n^2+\frac {\log n}n}}
			\geq
			1-n^{-c}.
			$$
		\end{proposition}
		In view of \prettyref{thm:finite-dim} we conclude that the neural network estimator $f_\phihat$ achieves the nonparametric error rate in terms of the population based error as well. The proof of the result is provided in \prettyref{app:proof-population-loss}. Our proof uses properties of covering number bounds on the neural network class $\calG(L,d,1,N,M,B)$ for uniformly controlling the variations in the estimator $f_\phihat$. Note that one can establish a similar relation between the error $\norm{\cdot}_n$ and $\norm{\cdot}_2$ norms using the arguments as in our proof of \prettyref{prop:population-empirical-connection}, with additional justification to control the covering numbers for the estimated variable selection matrix $\hat \Theta$. We omit the proofs here.
	\end{remark}
	
	\begin{remark}[Comparison of the risk bounds with regular nonparametric estimation guarantees]
		\label{rmk:smoothness}
		In the classical nonparametric regression literature (e.g., see \cite{stone1980optimal,stone1982global}) the primary problem of interest is estimating data generating functions that are $(\beta,C)$-smooth for some constant $\beta$, with features having a finite dimension $d$. In this setup,  $\gamma^*$ can take the value $\beta\over d$, which, under reasonable assumptions on the degree of smoothness $\beta$ (e.g., degree of differentiability $\beta$ of $f^*$ is a large enough constant) can be guaranteed to be nonvanishing as $n$ grows. This ensures that the nonparametric estimation rate established in \prettyref{thm:finite-dim} is nontrivial. On the other hand, when $d$ is growing with $n$, for constant values of $\beta$, we get that $\beta\over d$ can become vanishing. However, under the sparsity assumption $f^*(\bx)=f^*(\bx_{\calJ})$, the dimension adjusted smoothness parameter $\gamma^*$ is bounded from below by at least $\beta\over |\calJ|$, which ensures that the rate of estimation in \prettyref{thm:sparse} is vanishing whenever $|\calJ|$ is of constant order and $\beta$ is at least a large constant.

	\end{remark}
	
	\begin{remark}[Discussing the error upper bounds]
			The above results, \prettyref{thm:finite-dim} and \prettyref{thm:sparse}, establish minimax guarantees for deep neural network estimators under the multiplicative speckle noise model in both the low-dimensional and sparse settings. In the low-dimensional case, the proposed estimator achieves the convergence rate that coincides, up to logarithmic factors, with the minimax optimal rate for estimating hierarchical functions under the classical additive noise model  \cite{kohler2021rate}. In the sparse setting, where the regression function depends only on an unknown subset of the features, the $\ell_1$-penalized estimator automatically adapts to the underlying sparsity structure and achieves the same convergence rate even when the ambient dimension grows with the sample size. Thus, the proposed estimator avoids the curse of dimensionality by exploiting the intrinsic low-dimensional structure of the target function. These results demonstrate that, although the multiplicative speckle noise model requires substantially different analytical techniques than the additive noise model, it does not incur any loss in statistical efficiency for estimating the underlying signal.
		\end{remark}
		
		\begin{remark}[Effects of the additive noise level on our estimator construction]
			Note that a nonvanishing value of $\sigma_\tau$, which enters in the negative log-likelihood based objective function \eqref{eq:likelihood-full} as $f^{*2}(\cdot)+\sigma_\tau^2$ in the denominator, prevents the objective functions from becoming ill-conditioned and improves the numerical stability of the optimization algorithm. Our restriction for the theoretical analysis, which assumes that $\sigma_\tau$ is a constant, is primarily technical and it is connected to using existing approximation results for deep neural networks that are established for finite target functions. Extending these approximation guarantees to accommodate arbitrary values of $\sigma_\tau$ is left for future work.
			
		\end{remark}

	\subsection{Impossibility result for signal estimation under speckle noise}

	Next, we present a lower bound on the estimation error for functional estimation in multiplicative noise setting, that shows that the function recovery guarantees for our deep learning estimators are minimax optimal, up to logarithmic factors. To this end, we consider the data generating function class to be $\calH(d,l,\calP)$, with $\calP$ having the hardness parameter $\gamma^*$ defined in \eqref{eq:gamma*}. We make the following additional assumption that ensures that the feature vectors $(\bx_1,\dots,\bx_n)$ are well distributed throughout the domain.
	
	\begin{myassumption}
		\label{asmp:riemann-approx-revised}
		
		The feature vectors are supported on a bounded domain, which, without loss of generality, we take to be $[0,1]^d$. The deterministic design points $(\bx_1,\ldots,\bx_n)$ satisfy
		\[
		\sup_{f\in\mathcal H(d,l,\mathcal P)}
		\left|
		\frac1n\sum_{i=1}^n f^2(\bx_i)
		-
		\int_{[0,1]^d}f^2(\bx)\,d\bx
		\right|
		=o(1).
		\]
		
	\end{myassumption}
	
	\begin{remark}[Explanation of the assumption]
		The above assumption is satisfied in many standard settings. For example, if the design points form a regular Cartesian grid over $[0,1]^d$, then the empirical average is a Riemann sum, and the assumption follows for every continuous function, and hence for every $f\in\mathcal H(d,l,\mathcal P)$.
	\end{remark}

	In view of the above assumption, we establish the following minimax lower bound.
	
	\begin{theorem}[Minimax lower bound for nonparametric estimation under speckle noise]
		\label{thm:speckle-lb}
		Consider fixed $\bx_1,\dots,\bx_n$ satisfying \prettyref{asmp:riemann-approx-revised} and independent samples $\{(\bx_i,y_i)\}_{i=1}^n$ from the model \eqref{eq:model} with $\sigma_\tau=O(1)$ and $f^*:\reals^{d}\mapsto \reals^+$. Suppose that $\calJ$ denotes the set of active coordinates in $\bx$ (we assume $\calJ=\{1,\dots,d\}$ in the non-sparse analysis) and $f^*(\bx)=f^*(\bx_\calJ)\in \calH(|\calJ|,l,\calP)$, where $\calH(|\calJ|,l,\calP)$ has the smoothness parameter $\gamma^*$ as in \eqref{eq:gamma*}. Then, there exist constants $c_1>0,p^*\in (0,1)$ such that for all  $n$
		\begin{align*}
			\inf_{\hat f} \sup_{f^*(\bx_{\calJ})\in \calH(|\calJ|,l,\calP)} \PP\qth{\|\hat f-f^*\|_n^2 \ge c_1 n^{-\frac{2\gamma^*}{2\gamma^*+1}}}>p^*.
		\end{align*}
	\end{theorem}
	The proof of the above result is presented in \prettyref{sec:proof-lower-bound}.
	
	\begin{remark}[Optimality of the DNN-estimators]
		The above results show that the proposed DNN estimators achieve estimation rates that are within logarithmic factors (in the sample size $n$) of the minimax-optimal thresholds. Consequently, the statistical difficulty of nonparametric estimation under multiplicative speckle noise is essentially the same as that under the corresponding additive noise model. The additional logarithmic factors appearing in the upper bounds of \prettyref{thm:finite-dim} and \prettyref{thm:sparse} are standard artifacts in the analysis of deep neural network estimators and arise from controlling the complexity of the DNN function class \cite{kohler2021rate}. It is plausible that these logarithmic factors can be removed, or at least further reduced, by employing more classical nonparametric estimators, such as local polynomial estimators, whose analysis can often be carried out more precisely. However, such an investigation is beyond the scope of the present work, as our primary objective is to characterize the statistical limits of despeckling using DNN estimators.
	\end{remark}
	
	\begin{remark}[Dependency of the minimax rates on the additive noise level]
		The existing literature indicates that the minimax risk for estimating $f^*$ is a nondecreasing function of the additive noise level $\sigma_\tau$, i.e., larger noise makes the estimation problem intrinsically more difficult. The proof follows the same argument as \cite[Lemma 3.1]{malekian2025speckle}; see also \cite[Lemma 4.2]{xing2025minimax}, and is therefore omitted. Consequently, we expect the estimation error bounds for the proposed DNN estimators to deteriorate as $\sigma_\tau$ increases. Since our current analysis is restricted to the case where $\sigma_\tau$ is bounded by a constant, a precise characterization of the dependence of the estimation error on $\sigma_\tau$ is left for future work.
		
	\end{remark}

	\section{Numerical studies}
	
	\label{sec:simulation}
	
	\subsection{Network training}
	
	\paragraph
	{Neural network hyperparameters}
	To estimate the true function, we employed the deep neural network construction prescribed in \prettyref{thm:finite-dim} and \prettyref{thm:sparse}. The loss function is optimized over the neural network class $\calG(L,d,1,N,M,B)$, where the data dimension $d$, number of hidden layers $L$, and the number of nodes $N$ in each hidden layer will be specified on a case by case basis. In the sparsity analysis, the number of hidden layers will also serve as the dimension of the sparsity inducing matrix $\Theta\in \reals^{d\times N}$. The truncation parameter $M$ is set at 9999 to accommodate for possible large values of the target functional. Although the chosen value of $M$ is sufficiently large that truncation rarely becomes active, it guarantees bounded intermediate representations as suggested in our theory.
	
	\paragraph{Training Procedure}
	The network parameters, including both the projection matrix $\Theta$ (in the sparsity based analysis only) and the weights of the deep neural network, were optimized jointly using the Adam optimizer with learning rate $10^{-3}$.  We chose the regularization coefficient $\lambda$ as
	$
	\lambda
	=
	\frac{\log(ndN)}{n},
	$
	which is the dominating term in the prescribed theoretical value in \prettyref{thm:sparse}. We fix the threshold parameter $\tau$ at $10^{-3}$. Training was performed using mini-batches of size 64. In the case of sparsity analysis, $\Theta$ was initialized with independent Gaussian entries scaled by 0.01.
	
	
	\subsection{Function estimation in multiple dimensions}
	\paragraph{Data generation} We show that our algorithm can effectively detect the sparse locations in the feature vector that construct the data generating signal. 
	We considered a multivariate nonlinear regression setting with sample sizes $n\in \{1000,1500,\dots,7000\}$ and number of features $d=6$. For each replication, the feature vectors \(\bx_i=(x_{i1},\ldots,x_{i6})^\top\,i=1,\ldots,n\), were generated independently from the uniform distribution on \([0.1,1]^{6}\). The response variable followed the model in \eqref{eq:model} with $\sigma_\tau=1$. We choose the data generating true function $f^*$ as
	\begin{align}
		\label{eq:m17}
		f^*(\bx)
		=
		1
		+
		x_1
		+
		\log |x_2+0.5|
		+
		\cos(x_3)
		+
		(x_4x_5)^2
		+
		x_6^6.
	\end{align}
	to incorporate several distinct nonlinear effects, including logarithmic, trigonometric, interaction, and higher-order polynomial components. 
	
	\begin{figure}[t]
		\centering
		\begin{subfigure}[t]{0.49\linewidth}
			\caption{MSE plot against sample sizes}
			\includegraphics[width=0.9\linewidth]{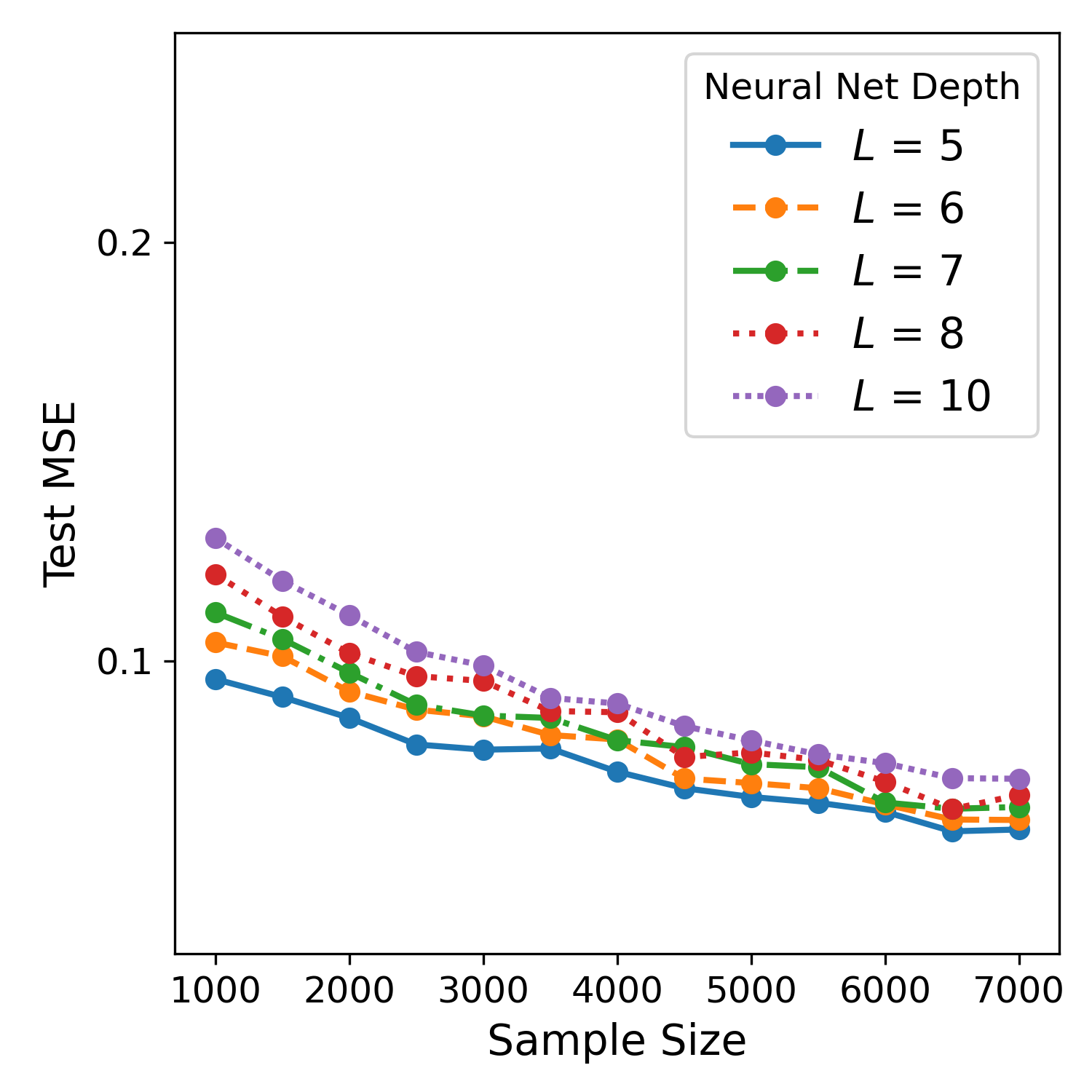}
			\label{fig:mse_vs_samples}
		\end{subfigure}
		\hfill
		\begin{subfigure}[t]{0.49\linewidth}
			\caption{MSE plot against epochs}
			
			\includegraphics[width=0.9\linewidth]{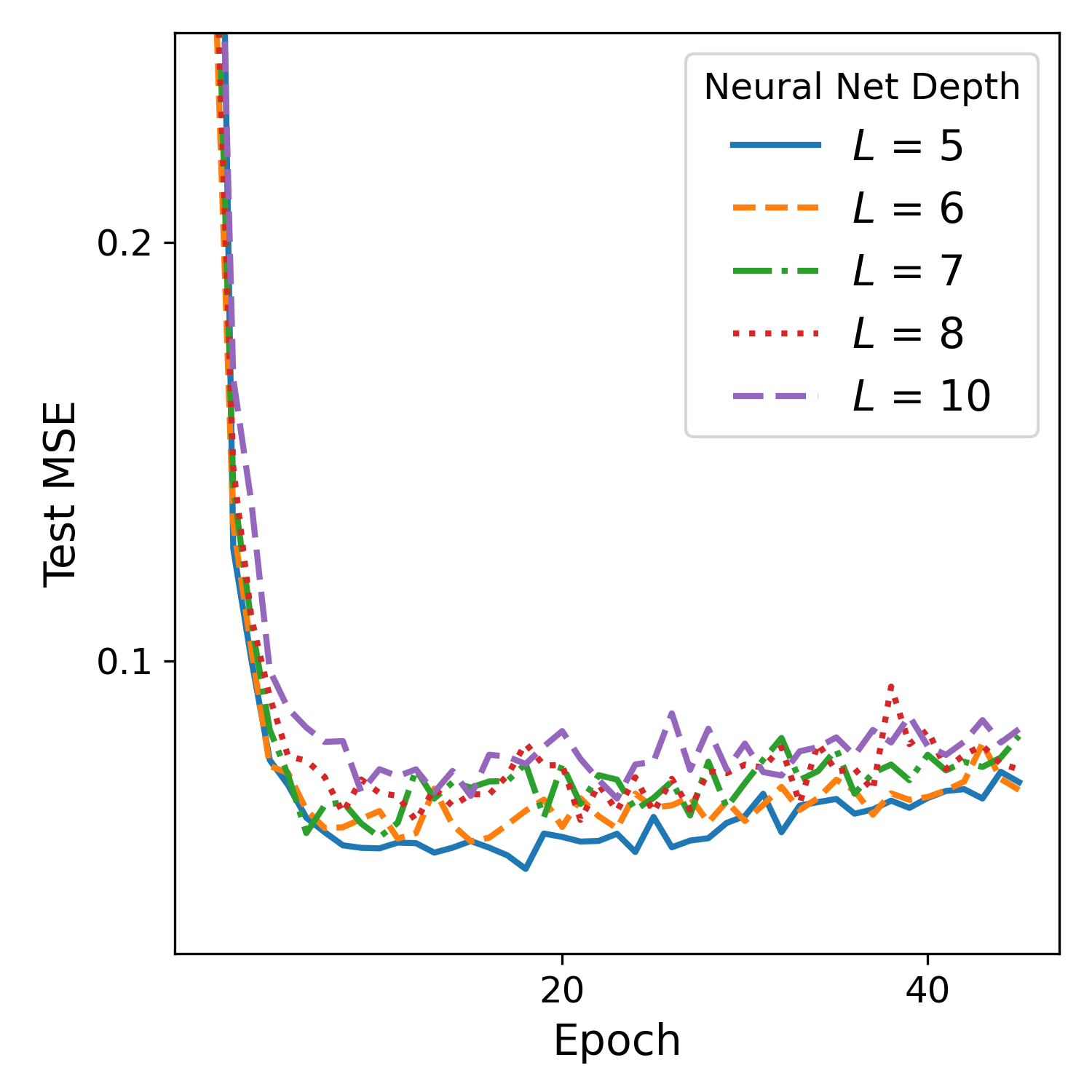}
			\label{fig:mse_vs_epoch}
		\end{subfigure}
		\caption{Comparison of error across different depths with fixed width 50. (a) Test MSE curves across different sample sizes and 50 epochs. (b) A comparison of Test  MSE values as the epoch increases, and sample size fixed at 4000.}
		\label{fig:mse}
	\end{figure}

	\paragraph{Experiment details} To standardize the comparison of neural networks, we fixed the network width at $N=50$ and used a baseline depth of $L=5$. The choice of $N=50$ was based on preliminary hyperparameter tuning, which indicated that increasing the width beyond 50 produced negligible improvements in predictive performance while incurring additional computational cost.
	
	In the first experiment presented in \prettyref{fig:mse_vs_samples}, we varied the network depth $L$ over the set $\sth{5,6,7,8,10}$. For each network configuration, we performed 500 independent Monte Carlo replications. In each replication, the network was trained for 50 epochs, and the mean squared error (MSE) with respect to the true data-generating function was computed. The choice of 50 training epochs was motivated by preliminary experiments, which showed that the training error decreased rapidly during the first 20--30 epochs and then exhibited only marginal improvements thereafter. We therefore used 50 epochs to ensure convergence while allowing training to continue beyond the point at which the error had largely stabilized. The MSE values were then averaged across the replications to obtain the Average MSE, which was plotted against different network depths and sample sizes. The results show that the performance differences among the various network depths diminish as the sample size increases, indicating that learning improves consistently across all architectures.
	
	In the second experiment, we fix the sample size to 4000 and the network width at 50 while varying the network depth to examine the evolution of the test error during training. We plot the average test MSE over 400 Monte Carlo replications as a function of the number of training epochs. As shown in \prettyref{fig:mse_vs_epoch}, the test MSE decreases rapidly during the initial epochs, reaches its lowest value after approximately 10--20 epochs, and then increases slightly, indicating mild overfitting with prolonged training. Importantly, this behavior is consistent across all network depths, suggesting that the networks exhibit similar generalization performance regardless of depth.

	\subsection{Active coordinate detection with high-dimensional features}
	
	\paragraph{Data generation}
	We demonstrate that our algorithm can effectively identify the sparse locations in the feature vector that generate the underlying signal. We consider a high-dimensional nonlinear regression setting with $n=4000$ observations and $d=400$ features. For each Monte Carlo replication, the feature vectors $\bx_i=(x_{i1},\ldots,x_{id})^\top$, $i=1,\ldots,n$, were generated independently from the uniform distribution on $[0.1,1]^{400}$. The response variable followed the model in \eqref{eq:model} with $\sigma_\tau=1$. We chose the true data-generating function $f^*$ as in \eqref{eq:m17}, which incorporates diverse nonlinear effects, including logarithmic, trigonometric, interaction, and higher-order polynomial terms. Importantly, only the first six covariates contributed to the response, whereas the remaining 394 variables were irrelevant noise features. To investigate the effect of a stronger signal on variable selection, we also considered the amplified signal
	$
	f^{**}(\bx)
	=4+10\left(x_1+\log|x_2+0.5|+\cos(x_3)+(x_4x_5)^2+x_6^6\right).
	$
	
	\begin{figure}[t]
		\centering
		
		\begin{subfigure}[t]{\textwidth}
			\centering
			
			\caption{Average estimated $\Theta^\top$ with $f^*(\bx)
				=
				1
				+x_1
				+
				\log |x_2+0.5|
				+
				\cos(x_3)
				+
				(x_4x_5)^2
				+
				x_6^6$.}\includegraphics[width=\linewidth]{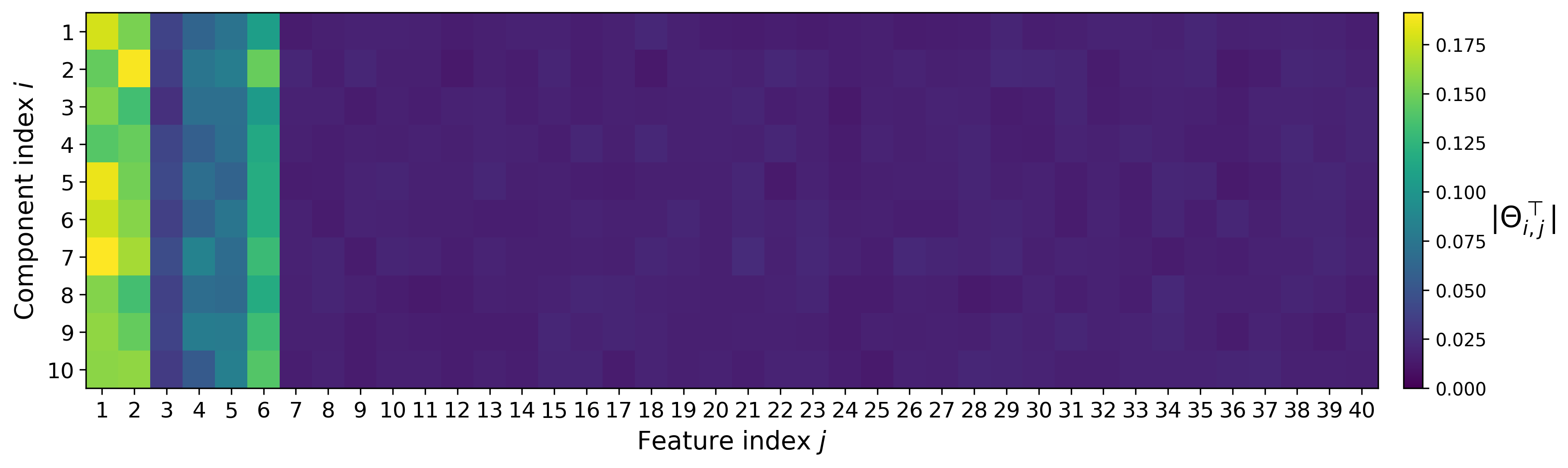}
			\label{fig:importance_heatmap}
		\end{subfigure}
		
		\begin{subfigure}[t]{\textwidth}
			\centering
			
			\caption{Comparison for $\Theta^\top$ against amplified signal $f^{**}(\bx)
				=
				4
				+
				10(x_1
				+
				\log |x_2+0.5|
				+
				\cos(x_3)
				+
				(x_4x_5)^2
				+			x_6^6)$.}\includegraphics[width=\linewidth]{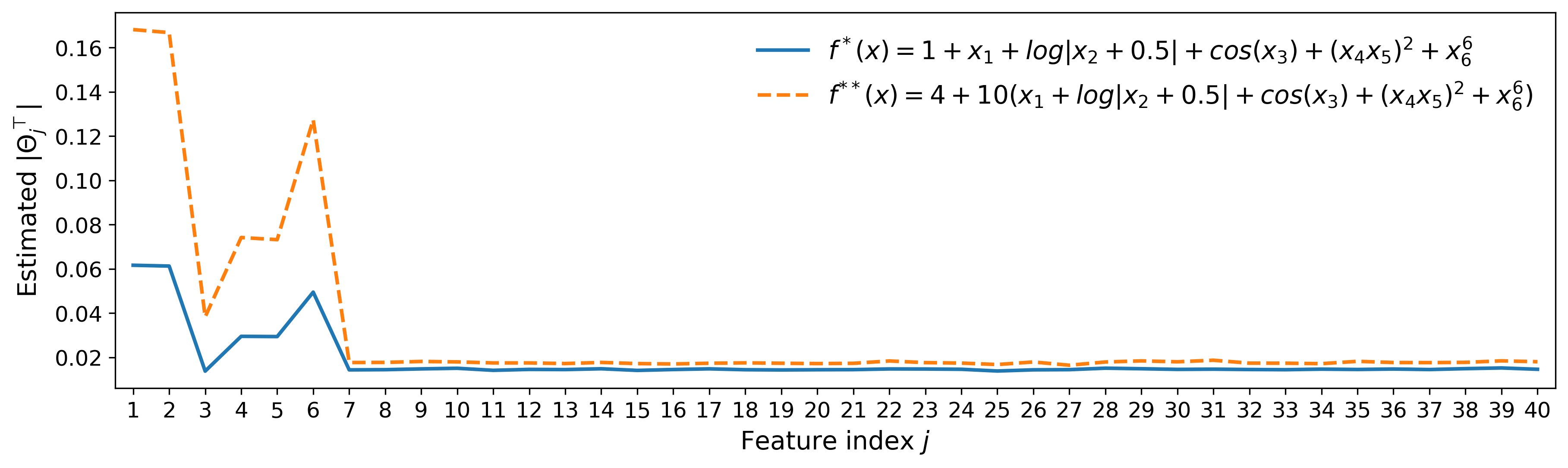}
			\label{fig:importance_second}
		\end{subfigure}
		
		\caption{Visualization of the estimated feature selection matrix $\Theta^\top$, showing that the DNN estimator efficiently detects the active feature coordinates. Regions with lighter shades correspond to larger entries in estimated $\Theta^\top$. (a) Heatmap of the average estimated entries for the first 40 features and first 10 rows. (b) A comparison of averaged estimated $|\Theta_j^\top|$ values for different features $j$ against an amplified signal.}
		\label{fig:importance}
	\end{figure}

	\paragraph{Experiment details}
	The experiment was repeated over 200 independent Monte Carlo replications. For each replication, we fixed the network depth and width at $L=5$ and $N=50$, respectively, and trained the network for 100 epochs. Compared with the low-dimensional experiments, we used a larger number of training epochs because optimization becomes more challenging in the high-dimensional setting, requiring additional iterations for the network to adequately learn the underlying structure. For each replication, the estimated projection matrix $\widehat{\Theta}$ was recorded, and the average absolute projection weights were computed across all replications. These averaged absolute weights serve as a measure of variable importance and were used to assess the ability of the sparse projection layer to recover the truly relevant predictors. The results in \prettyref{fig:importance} show that the largest projection weights are concentrated on the first six features, which coincide with the true active coordinates in the data-generating function \eqref{eq:m17}. This demonstrates that the proposed method successfully identifies the relevant feature locations despite the presence of a large number of irrelevant covariates. Furthermore, when the signal is amplified to $f^{**}(\bx)$, the distinction between the active and inactive features becomes even more pronounced, resulting in clearer identification of the relevant variables.

	\section{Proof overview of \prettyref{thm:finite-dim} via quantifying likelihood differences}
	\label{sec:likelihood-diff}

	We present an overview of the proof of \prettyref{thm:finite-dim} here. Our proof relies on utilizing the fact that the log-likelihood for the maximum likelihood estimator is higher than that of the true data generating function, and we convert the difference between the corresponding log-likelihood values to a squared error distance between the true and estimated regression functions. Before presenting the details we  note that, under the assumption of a known $\sigma_\tau=O(1)$, the proof remains essentially the same when we consider the special case $\sigma_\tau =0$, and establish risk guarantees for the estimator based on the revised negative log-likelihood
	\begin{align}
		\label{eq:likelihood-reduced}
		\ell(f)
		=
		\frac 1{n} \sum_{i=1}^n \frac{y_i^2}{\{f(\bx_i)\}^2}
		+
		\frac 1n \sum_{i=1}^n \log\bigl(\{f(\bx_i)\}^2\bigr).
	\end{align}
	This follows as for a general $\sigma_{\tau}$, the optimizer $\hat f$ of \eqref{eq:likelihood-reduced} guarantees that $\hat f^2$ is consistent for estimating $f^{*2}+\sigma_\tau^2$, and then $f^*$ can be estimated using $\sqrt{\hat f^2-\sigma_\tau^2}$. As $f^{*2}+\sigma_\tau^2$ has the same smoothness as $f^{*2}$, the theoretical analysis for the optimizer of \eqref{eq:likelihood-reduced} is similar to the analysis of the optimizer of \eqref{eq:likelihood-full}. Hence, for the rest of the paper we present the error guarantees for deep learning optimizer of \eqref{eq:likelihood-reduced}, under the model without additive noise (i.e., $\sigma_\tau=0$)
	\begin{align}
		\label{eq:model-reduced}
		y_i = f^*(\bx_i)\xi_i,
		\quad 
		i=1,\dots,n,\quad
		\xi_i\simiid N(0,1).
	\end{align} 
	
	Now we present the high-level argument for proving the risk guarantee of the DNN optimizer of \eqref{eq:likelihood-reduced}. By taking expectation of the likelihood \eqref{eq:likelihood-reduced} with respect to the model \eqref{eq:model-reduced}, we get
	
	\begin{align}
		\label{eq:expected_likelihood}
		\bar \ell (f)
		=
		\EE\qth{\ell(f)\mid \bx_1,\dots,\bx_n}
		&=
		\frac 1{n} \sum_{i=1}^n \frac{\{f^*(\bx_i)\}^2}{\{f(\bx_i)\}^2}
		+
		\frac 1n \sum_{i=1}^n \log\bigl((f(\bx_i))^2\bigr).
	\end{align}
	Using $f=f^*$, the last display implies
	\begin{align}
		\label{eq:cond_expectation}
		\bar \ell (f^*)
		=
		\EE\qth{\ell(f^*)\mid \bx_1,\dots,\bx_n}
		&=
		1 + \frac 1n \sum_{i=1}^n \log\bigl(f^*(\bx_i)^2\bigr).
	\end{align}
	From the above display, we get that for any function $f$, we have
	\begin{align}
		\label{eq:expected-likelihood-diff}
		\bar\ell(f)
		-\bar\ell(f^*)
		=
		\frac 1n \sum_{i=1}^n
		\left(
		\frac{\{f^*(\bx_i)\}^2}{\{ f(\bx_i)\}^2}-1
		-\log \left(\frac{f^*(\bx_i)^2}{ f(\bx_i)^2}\right)
		\right)
	\end{align}
	Reordering the terms in the expansion of \eqref{eq:likelihood-reduced}, for any function $f$ we get
	\begin{align}
		\ell( f) - \ell(f^*)
		&=
		\frac{1}{n} \sum_{i=1}^n
		\left(
		y_i^2 \left( \frac{1}{ f(\bx_i)^2} - \frac{1}{f^*(\bx_i)^2} \right)
		\right)
		+
		\frac 1n \sum_{i=1}^n
		\log\!\left(
		\frac{ f(\bx_i)^2}{f^*(\bx_i)^2}
		\right)
		\label{eq:main_expansion}
		\\
		&=
		\frac{1}{n} \sum_{i=1}^n
		\left(
		y_i^2 \left(\frac{1}{ f(\bx_i)^2} - \frac{1}{f^*(\bx_i)^2}\right)
		-
		f^*(\bx_i)^2 \left( \frac{1}{ f(\bx_i)^2} - \frac{1}{f^*(\bx_i)^2} \right)
		\right)
		\notag \\
		&\quad
		+
		\frac 1{n} \sum_{i=1}^n \pth{\frac{\{f^*(\bx_i)\}^2}{\{ f(\bx_i)\}^2}
		- 1- \log\Bigl({f^*(\bx_i)^2\over f(\bx_i)^2}\Bigr)}
		\notag.
	\end{align}
	We use \eqref{eq:expected-likelihood-diff} and a quadratic formulation to rewrite the last display as
	\begin{align}
		\label{eq:likelihood-difference}
		\bar \ell(f)-\bar\ell(f^*)
		=
		{
			\by^\top (\Sigma^*-\Sigma_f) \by
			-
			\E \big[ \by^\top (\Sigma^*-\Sigma_f) \by \big]
			\over n}
		+
		\ell(f) - \ell(f^*),
	\end{align}
	where
	\begin{align}
		\label{eq:definitions}
		\by = [y_1,\dots,y_n], \quad
		\Sigma_f = \diag\!\left(
		\frac{1}{f(\bx_1)^2}, \dots, \frac{1}{f(\bx_n)^2}
		\right), \quad
		\hat\Sigma = \Sigma_{\hat f}, \quad
		\Sigma^* = \Sigma_{f^*}.
	\end{align}
	
	Consider the parametrization $f_\phi$ of $\calG(L,d,1,N,M,B)$ via its set of weights and biases $\phi$. We choose $f_{\phistar}$ as the best DNN from $\calG(L,d,1,N,M,B)$ closest to the true data generating function $f^*$. In view of the following result we note that $\phistar$ can be chosen in such a way that
	\begin{align}
		\label{eq:approx-error}
		\|f_{\phistar}-f^*\|_\infty < N^{-2\gamma^*},
	\end{align}
	where $\gamma^*$ is the degree adjusted smoothness defined in \prettyref{def:dim-adjusted-smoothness}. 
	\begin{lemma}\cite[Approximating $\calH(d, l,\calP)$ via deep \relu Networks]{fan2024factor}
		\label{lmm:approx-smooth}
		Let $g$ be a $d$-variate, $(\beta, C)$-smooth function. There exists some universal constants $c_1$--$c_5$ depending only on $d, \beta, C$, such that for arbitrary $N\in \mathbb{N}^+ \setminus \{1\}$, there exists a deep \relu network $g^\dagger \in \mathcal{G}(c_1, d, 1, c_2 N, \infty, c_3 N^{c_4})$ satisfying $\|g^\dagger-g\|_{\infty,[0,1]^d} \le c_5 N^{-2\beta/d}$.
		Furthermore, if $g\in \mathcal{H}(d, l,\mathcal{P})$ with $\sup_{(\beta, t)\in \mathcal{P}} (\beta \lor t) <\infty$ and $g$ is supported on $[-c_6, c_6]^d$ for some constant $c_6$. There also exists some universal constants $c_{7}$--$c_{11}$ such that for arbitrary $N\in \mathbb{N}^+ \setminus \{1\}$, there exists a deep \relu network $g^\dagger \in \mathcal{G}(c_7, d, 1, c_8 N, \infty, c_9 N^{c_{10}})$ satisfying $\|g^\dagger-g\|_{\infty,[-c_{6},c_6]^d} \le c_{11} N^{-2\gamma^*}$.
	\end{lemma}
	Note that we will later choose $N$ as in \prettyref{asmp:model-and-hyperparameters} to guarantee that the error of approximating $f^*$ by $f_\phistar$ uniformly decays to zero as the sample size $n$ increases. To this end, consider the maximum likelihood optimizer $f_\phihat$ of \eqref{eq:likelihood-reduced}
	\begin{align}
		\label{eq:estimator-reduced}
		f_\phihat \in \argmin_{f_\phi\in \calG(L,d,1,N,M,B)} \ell(f_\phi).
	\end{align}
	Note that the above implies
	\begin{align}
		\ell(f_\phihat) \leq \ell(f_\phi),
		\quad 
		f_\phi\in \calG(L,d,1,N,M,B).
	\end{align}
	Then using \eqref{eq:likelihood-difference} we get from $\ell(f_\phihat) - \ell(f_\phistar) \leq 0$, with $\Sigma^*,\hat \Sigma$ as defined in \eqref{eq:definitions}
	\begin{align}
		\label{eq:m1}
		\bar\ell(f_\phihat)
		-\bar\ell(f^*)
		\leq
		\frac{
			\by^\top (\Sigma^* - \hat\Sigma) \by
			-
			\E \big[ \by^\top (\Sigma^* - \hat\Sigma) \by \big]
		}{n}
		+
		\ell(f_{\phi^*})-\ell(f^*).
	\end{align}
	We would like to convert the above inequality into an upper bound on $\|f_\phihat-f^*\|_n^2$. However, the dependency of $f_\phihat$ on the data prevents us from directly using standard concentration bounds, as they often require a fix $f_\phihat$. To circumvent this issue we resort to a covering number argument and union bounds. 
	Consider a $\delta$-covering of the DNN class $\calG(L,d,1,N,M,B)$ with respect to the maximum norm defined in \eqref{eq:metric-neural-net}
	\begin{align}
		\label{eq:covering}
		\calN(\delta,\calG,\norm{\cdot}_\infty)
		=
		\Biggl\{
		\sth{f_{\phi_i}}_{i=1}^m:
		\min_{j=1}^m \|\phi-\phi_j\|_{\infty}\leq \delta \text{ for any $f_\phi\in \calG(L,d,1,N,M,B)$.}
		\Biggr\}
	\end{align}
	In particular, we will pick $\delta=\frac 1{n^{L+8}}$. 
	Using 
	$$\ell(f_\phitilde)
	= \ell(f_\phihat) + \ell(f_\phitilde)-\ell(f_\phihat)
	\leq \ell(f_\phistar) + \ell(f_\phitilde)-\ell(f_\phihat)$$ 
	we use the notation $\tilde \Sigma =\Sigma_{f_\phitilde}$ as in \eqref{eq:definitions}, to derive from \eqref{eq:m1}
	\begin{align}
		\label{eq:ineq-main}
		\bar\ell(f_\phitilde)
		-\bar\ell(f^*)
		&\leq
		\frac{
			\by^\top (\Sigma^* - \tilde\Sigma) \by
			-
			\E \big[ \by^\top (\Sigma^* - \tilde \Sigma) \by \big]
		}{n}
		+
		\ell(f_{\phi^*})-\ell(f^*)
		+\ell(f_{\phitilde})- \ell(f_\phihat).
	\end{align}
	The rest of the proof first performs an analysis for fixed $f_\phitilde\in \calN(\delta,\calG,\norm{\cdot}_\infty)$, and then uses a union bound to convert our result into an error for $f_\phihat$. The remainder of the proof can be divided into controlling the three major error quantities given as follows.
	\begin{itemize}
		
		\item {\it Estimation error:} We use results from empirical process theory and the Hanson-Wright inequality given in \prettyref{lmm:hanson wright} to bound
		$
		\by^\top(\Sigma^*-\tilde\Sigma)\by
		-
		\mathbb{E}\!\left[
		\by^\top(\Sigma^*-\tilde\Sigma)\by
		\right],
		$
		uniformly over $f_\phitilde \in \calN(\delta,\calG,\norm{\cdot}_\infty)$.
        This is a crucial step that connects the above centered random variable with $\|f^*-\tilde f\|_n^2$.
		\begin{lemma} \cite[Hanson-Wright inequality]{hansonwright71} \label{lmm:hanson wright} 	
			Let $\bxi= (\xi_1,...,\xi_n)^\top $ be a random vector with independent components with $\E[\xi_i]=0$ and $\|\xi_i\|_{\rm subgau}\leq \tilde K$. Let $D$ be an $n\times n$ matrix and $\norm{D}_{HS}$ denote its Hilbert-Schmidt norm. Then there is a constant $\breve c>0$ such that, for $t>0$, 
			\begin{align*}
				\Prob \qth{{\Big|\bxi^\top  D\bxi-\E[\bxi ^\top  D\bxi]\Big|>t}}  \leq 2\exp\left(-\breve c\min\left({t^2\over \tilde K^4\|D\|_{HS}^2 },{t \over \tilde K^2\|D\|_2 }\right)\right).
			\end{align*}
		\end{lemma}
		
		\item {\it Approximation error:} We control the term $\ell(f_{\phi^*}) - \ell(f^*)$ by establishing that the likelihood function $\ell(\cdot)$ is Lipschitz, and then use that $f^*$ is well approximated by $f_\phistar$, as in  \eqref{eq:approx-error}. We also control $\ell(f_{\phitilde}) - \ell(f_\phihat)$ using the Lipschitz property of $\ell(\cdot)$ and the bound $\|\phitilde - \phihat\|_{\infty}\leq \delta$.

		\item {\it Difference of expected likelihood:} We establish a lower bound to $\bar\ell(f_\phitilde)
		-\bar\ell(f^*)$ in terms of $\|f_\phitilde-f^*\|_n^2$, using a second order expansion of $\bar\ell$.

	\end{itemize}
	
	Finally, we use $f_\phitilde\in \calN(\delta,\calG,\norm{\cdot}_\infty)$ that is closest to $f_\phihat$ to combine the above bounds and extract an error upper bound on $\|f_\phihat-f^*\|_n^2$. The readers are referred to \prettyref{app:theory-multidim} for the technical details. The analyses of the sparse setting is similar to the above strategy and presented in \prettyref{app:theory-sparse}.

	\bibliographystyle{apalike}
	\bibliography{reference_npspeckle}
	
	\appendices

	\section{Proof of \prettyref{thm:finite-dim}}
	
	\label{app:theory-multidim}

	We prove \prettyref{thm:finite-dim} by completing the steps outlined at the end of \prettyref{sec:likelihood-diff}. Throughout our proof we use $\tilde C,\breve C,\bar C$, with different subscripts, as positive constants depending on $f_{\min},f_{\max},M$, as necessary.

	\subsection{Controlling the estimation error}
	
	Next we will use results from the empirical process theory to bound
	$
	\by^\top(\Sigma^*-\tilde\Sigma)\by
	-
	\mathbb{E}\!\left[
	\by^\top(\Sigma^*-\tilde\Sigma)\by
	\right],
	$
	for any fixed neural net $f_\phitilde\in \calN(\delta,\calG,\norm{\cdot}_\infty)$.
	In view of \prettyref{lmm:empirical-bound-1}, we get that for any fixed neural net $f_\phitilde\in \calN(\delta,\calG,\norm{\cdot}_\infty)$, with probability $1-e^{-t}$ we have
	\begin{align}
		\label{eq:m22}
	\left| \by^\top \tilde{B} \by - \mathbb{E}[\by^\top \tilde{B} \by] \right|
	\le C_1 \sqrt{t}\,\|\tilde B\|_{HS} + t \|\tilde B\|_2 .
	\end{align}
	where
	\begin{align}
		\label{eq:tilde-B}
	\tilde B = \tilde \Sigma^{-1}(\tilde \Sigma - \Sigma^*) \tilde\Sigma^{-1}(\tilde \Sigma - \Sigma^*),
	\quad
	\tilde \Sigma = \Sigma_{f_\phitilde}
	\text{ as in \eqref{eq:definitions}}.
	\end{align}
	To bound the size of $\calN(\delta,\calG,\norm{\cdot}_\infty)$ from above, we use the following Lipschitz property of $f_\phi$.
	\begin{lemma}\cite[Lemma 8 with $\bTheta$ as an Identity matrix and no diversified projection component]{fan2024factor}
		\label{lmm:lipschitz}
		Given two neural net parameters $\phi=\sth{(\bW_\ell,\bb_\ell)}_{\ell=1}^{L+1},\breve\phi=\{(\breve\bW_\ell,\breve\bb_\ell)\}_{\ell=1}^{L+1}$ from $\calG(L,d,1,N,M,B)$, we have for a constant $C(M,K,L)>0$ depending on $M,K,L$ that
		$$\sup_{\bx\in [-K,K]^d}|f_{\phi}(\bx)-f_{\breve\phi}(\bx)| \le C(M,K,L)B^L(N+1)^{L+1}\|\phi-\breve\phi\|_\infty.$$ 
	\end{lemma}
	Given any DNN in the class $\calG(L,d,1,N,M,B)$, its number of parameters is given by 
	\begin{itemize}
		\item $L$-many $N\times N$ matrices of weights, 
		\item $L$-many $N\times 1$ vectors of the bias parameters which construct the intermediate layers, 
		\item at most $2N$ parameters that connect the output with the final layer, and 
		\item $Nd+N$ weights/biases that connects the first layer with the $d$-dimensional input vector.
	\end{itemize} 
	Hence, the total number of parameters are bounded from above by
	\begin{align}
		\label{eq:total-parameters}
		A = N^2 L + NL + 2N + Nd+N=N^2L+N(L+d+3).
	\end{align}
	Then, the size $\delta$-covering $\calN(\delta,\calG,\norm{\cdot}_\infty)$ of the class $\calG(L,d,1,N,M,B)$ can be bounded as follows
	\begin{align}
		\label{eq:covering-bound}
		|\calN(\delta,\calG,\norm{\cdot}_\infty) |
		\leq 
		\left(\frac{CB^L(N+1)^{L+1}}{\delta}\right)^A,
		\quad C=C(M,K,L),
	\end{align}
	where we used \prettyref{lmm:lipschitz} and recall from \prettyref{asmp:model-and-hyperparameters} that $\bx$ has bounded coordinates. For the rest of the work, we will use the above covering number bound with $N\approx \pth{n / \log n}^{\frac 1{4\gamma^*+2}}$ and the $\delta$-choice
	\begin{align}
		\label{eq:delta-value}
		\delta=\frac 1{n^{(c_4+1)L+8}}.
	\end{align}
	for a large enough constant $c^4$ such that $\log B\leq c_4\log n$, as provided in \prettyref{asmp:model-and-hyperparameters}\ref{pt:hyperparam}. Then, we get from \eqref{eq:m22} that with probability $1 - e^{- t}$ for some $c_1>0$ ($A$ is as defined in \eqref{eq:total-parameters})
	\begin{align}
		\label{eq:m23}
		\left|\by^\top \tilde B \by - \mathbb{E}\qth{\by^\top \tilde B \by}\right|
		\le C_1\sqrt{\log |\calN(\delta,\calG,\norm{\cdot}_\infty) | + t}\cdot \|\tilde B\|_{HS} + (\log |\calN(\delta,\calG,\norm{\cdot}_\infty) |+t)\|\tilde B\|_2,
		 \end{align}
	for all $f_\phitilde\in \calN(\delta,\calG,\norm{\cdot}_\infty)$. Next, note that for some constant $C(f_{\min},f_{\max},M)>0$, we have that
	\[
	\|\tilde B\|_{HS} \le \sqrt{\calL},
	\qquad
	\|\tilde B\|_2 
	=
	\max_{i}
	{(\tilde f(\bx_i))^4}\pth{\frac 1{\tilde f^2(\bx_i)}-\frac 1{f^{*2}(\bx_i)}}^2
	\le 
	C(f_{\min},f_{\max},M),
	\]
	where we define
	\begin{align}
		\label{eq:calL}
		\calL = \mathrm{Tr}\!\left(
		\Sigma^{*-1}(\tilde\Sigma-\Sigma^*)\Sigma^{*-1}(\tilde\Sigma-\Sigma^*)
		\right)
		=\sum_{i=1}^n 
		(f^*(\bx_i))^2\pth{\frac 1{f_\phitilde(\bx_i)^2}-
			\frac 1{f^*(\bx_i)^2}}^2.
	\end{align}	
	In view of \eqref{eq:m22}, the above display implies that for some constant $C_1>0$, with probability $1-e^{-t}$
	\begin{align}
		\label{eq:ineq-estimation}
		&\abs{\by^\top(\Sigma^*-\tilde\Sigma)\by
			-
			\mathbb{E}\!\left[
			\by^\top(\Sigma^*-\tilde\Sigma)\by
			\right]}
			\notag\\
		&\leq C_1\sth{
			\sqrt{\log |\calN(\delta,\calG,\norm{\cdot}_\infty) | + t}\cdot\,\sqrt{\calL}
			+ (\log |\calN(\delta,\calG,\norm{\cdot}_\infty) | + t)}
	\end{align}
	We note that $\log |\calN(\delta,\calG,\norm{\cdot}_\infty) | \leq A\log n$ and make the following choice for $t$
	\[
	t = c_2A\log n
	\]
	for a large constant $c_2>0$, which guarantees, for some constant $c>2$, we have with probability $1-\frac{1}{n^c}$
	\begin{align}
		\label{eq:m7}
		\left|\by^\top \tilde B \by - \mathbb{E}(\by^\top \tilde B \by)\right|
		\le C_1(\sqrt{A\log n}\,\sqrt{\calL}
		+ (A\log n)),
		\quad
		\text{for all } f_\phitilde\in\calN(\delta,\calG,\norm{\cdot}_\infty).
	\end{align}	
	
	\subsection{Approximation Error}
	
	We will analyze the term $\ell(f_{\phi^*}) - \ell(f^*)$. In view of \eqref{eq:likelihood-difference} we get
	\begin{align}
		\label{eq:m2}
		\ell(f_{\phi^*}) - \ell(f^*)
		&= \frac{1}{n}\Big[\by^\top(\Sigma^{\phi^*}-\Sigma^*)\by - \mathbb{E}[\by^\top(\Sigma^{\phi^*}-\Sigma^*)\by]\Big]
		+ \bar{\ell}(f_{\phi^*}) - \bar{\ell}(f^*)
	\end{align}
	Using the fact that for all large $n$, in view of \prettyref{asmp:model-and-hyperparameters} and \eqref{eq:approx-error}, $\|f^*-f_{\phistar}\|_\infty$ converges to zero as $n$ increases, we have for some constant $C_1(f_{\min},f_{\max},M)>0$
	$$
	\abs{{f^*(\bx_i)^2\over f_\phistar(\bx_i)^2}-1}
	\leq 
	C_1(f_{\min},f_{\max},M)\|f^*-f_\phistar\|_\infty
	\leq 0.2,
	\quad
	i=1,\dots,n,
	\quad 
	n\geq C_3(f_{\min},f_{\max},M),
	$$
	for some constant $C_3(f_{\min},f_{\max},M)>0$ depending on $f_{\min},f_{\max},M$. In view of the above, we use
	\[
	z - 1 - \log z \le 2(z-1)^2 \quad \text{for } 0.8 \le z < 1.2,
	\]
	to get that
	\begin{align}
		\label{eq:m4}
		\bar{\ell}(f_{\phi^*}) - \bar{\ell}(f^*)
		&= \frac{1}{n}\sum_{i=1}^n
		\left[
		\frac{f^*(\bx_i)^2}{f_{\phi^*}(\bx_i)^2}
		-1
		-\log\!\left(\frac{f^*(\bx_i)^2}{f_{\phi^*}(\bx_i)^2}\right)
		\right] 
		\nonumber\\
		&\le \frac{2}{n}\sum_{i=1}^n
		\left(
		\frac{f^*(\bx_i)^2}{f_{\phi^*}(\bx_i)^2}-1
		\right)^2
		\le
		2\|f^*-f_\phistar\|_\infty^2
		\cdot
		\frac{1}{n}\sum_{i=1}^n
		\frac{(f^*(\bx_i)+f_{\phistar}(\bx_i))^2}{f_{\phi^*}(\bx_i)^4}.
	\end{align}
	To bound the first summand in \eqref{eq:m2} we use the following result. The proof is provided in \prettyref{app:technical} 
	\begin{lemma}\label{lmm:empirical-bound-1}
		Consider the noiseless model $y_i=f^*(\bx_i)\xi_i,i=1,\dots,n,$ with $\xi_i\simiid N(0,1)$. Define
		\[
		\tilde{\Sigma}
		=
		\{\mathrm{diag}((\tilde f(\bx_1))^2,\dots,(\tilde f(\bx_n))^2)\}^{-1},
		\qquad
		\Sigma^*
		=
		\{\mathrm{diag}\!\left((f^*(\bx_1))^2,\dots, (f^*(\bx_n))^2\right)\}^{-1},
		\]
		where $\tilde f$ is a function independent of the data.
		Then, there is a constant $C=C(f_{\min},f_{\max})$, such that with probability $1-e^{-t}$
		\[
		\big|\by^\top (\Sigmatilde-\Sigmastar) \by - \mathbb{E}[\by^\top (\Sigmatilde-\Sigmastar) \by]\big|
		\le C 
		\left(\sqrt{t \cdot \mathrm{Tr}\!\left(
			\Sigma^{*-1}(\tilde\Sigma-\Sigma^*)\Sigma^{*-1}(\tilde\Sigma-\Sigma^*)
			\right)}+ 
		t\|\tilde f - f^*\|_\infty
		\right).
		\]
	\end{lemma}
	
	The above, with $\tilde f = f_{\phi^*}$, implies for a constant $\breve C_4 = \breve C_4(f_{\min},f_{\max},M)>0$, with probability $1-e^{-t}$
	\begin{align}
		& \Big[\by^\top(\Sigma^{\phi^*}-\Sigma^*)\by - \mathbb{E}\big(\by^\top(\Sigma^{\phi^*}-\Sigma^*)\by\big)\Big]
		\leq 
		\sqrt{t\cdot 
			\sum_{i=1}^n
			\left(
			{\tilde f(\bx_i)^{-2}}
			-
			{f^*(\bx_i)^{-2}}
			\right)^2
			(f^*(\bx_i))^4
		}
		+
		t\cdot \|B\|_2
		\nonumber\\
		&=
		\sqrt{
			t\cdot 
			\sum_{i=1}^n
			\frac{(f_{\phi^*}(\bx_i)-f^*(\bx_i))^2(\tilde f(\bx_i)+f^*(\bx_i))^2}
			{(f^*(\bx_i))^2(\tilde f(\bx_i))^2}
		}
		+
		t\cdot \|B\|_2
		\nonumber\\
		&\le
		\breve C_4(\|f_{\phistar}-f^*\|_\infty
		\sqrt{nt}
		+
		t\cdot \|f_{\phistar}-f^*\|_\infty
		)
		=
		\breve C_4
		\|f_{\phistar}-f^*\|_\infty
		\left(
		\sqrt {nt} + t
		\right).
		\label{eq:m5}
	\end{align}
	Combining \eqref{eq:m4}, \eqref{eq:m5},  in view of \eqref{eq:m2}, we use $t=c\log n$ to get, for large $n$, with probability $1-\frac1{n^c}$ 
	\begin{align}
		\label{eq:ineq-approx}
		\big|\ell(f_{\phi^*}) - \ell(f^*)\big|
		&\le
		C_1 \|f_{\phi^*}-f^*\|_\infty^2
		+
		C_2{\|f_{\phi^*}-f^*\|_\infty}
		\left\{
		\sqrt {\log n\over n}
		+
		{\log n \over n}
		\right\}
		\notag\\
		&\stepa{\leq} 
		\tilde C_2 \pth{N^{-4\gamma^*}
			+N^{-2\gamma^*}\sqrt{\log n\over n}}
		\notag\\
		&\stepb{\leq} 
		\tilde C_3
		\pth{\pth{n\over \log n}^{-{4\gamma^*\over 4\gamma^*+2}}
			+\pth{n\over \log n}^{-{2\gamma^*\over 4\gamma^*+2}-\frac 12}}
		\notag\\
		&= \tilde C_3 \pth{n\over \log n}^{-{2\gamma^*\over 2\gamma^*+1}}
		\pth{1+\pth{n\over \log n}^{-{\gamma^*\over 2\gamma^*+1}-\frac 12+{2\gamma^*\over 2\gamma^*+1}}}
		\notag\\
		&= \tilde C_3 \pth{n\over \log n}^{-{2\gamma^*\over 2\gamma^*+1}}
		\pth{1+\pth{n\over \log n}^{-{1\over 2(2\gamma^*+1)}}}
		\leq 2\tilde C_3 \pth{n\over \log n}^{-{2\gamma^*\over 2\gamma^*+1}},
	\end{align}
	from \eqref{eq:approx-error}, (b) follows from our choice of $N$ in \prettyref{thm:finite-dim}.
	
	Next, we control $\ell(f_{\phihat}) - \ell(f_\phitilde)$. To this end, define
	\[
	\ell_i(f) = \frac{y_i^2}{f(\bx_i)^2} + \log f(\bx_i)^2 ,
	\]
	so that the quantity of interest can be written as
	\[
	\ell(f_{\phihat}) - \ell(f_\phitilde)
	=
	\frac{1}{n}\sum_{i=1}^n
	\Bigl(
	\ell_i(f_{\phihat})-\ell_i(f_{\phitilde})
	\Bigr).
	\]
	By the mean value theorem, for each $i$ there exists a value
	$\tilde f_i$ between $f_{\phihat}(\bx_i)$ and $f_{\phitilde}(\bx_i)$ such that
	\[
	\ell_i(f_{\phihat})-\ell_i(f_{\phitilde})
	=
	\ell_i'(\tilde f_i)
	\bigl(
	f_{\phihat}(\bx_i)-f_{\phitilde}(\bx_i)
	\bigr).
	\]
	Since
	$
	\ell_i(f)=\frac{y_i^2}{f^2}+\log f^2,
	$
	we obtain
	$
	\ell_i'(f)
	=
	-\frac{2y_i^2}{f^3}+\frac{2}{f}.
	$
	Because the network outputs are bounded away from zero,
	$f_\phi(x)\ge \frac 1M>0$, it follows that
	\[
	|\ell_i'(f)|
	\le
	{2M^3y_i^2}+ 2M .
	\]
	Using that $\|\phihat-\phitilde\|_\infty\le\delta$ and \prettyref{lmm:lipschitz}, we get that
	\begin{align}
		\label{eq:m29}
	|f_{\phihat}(\bx)-f_{\phitilde}(\bx)| \le C(M,K,L)B^L(N+1)^{L+1}\|\phihat-\phitilde\|_\infty
	\end{align}
	from \prettyref{lmm:lipschitz}, we combine the previous displays to get
	\begin{align}
		\label{eq:m8}
		|\ell(f_{\phihat}) - \ell(f_\phitilde)|
		\le
		\frac{1}{n}\sum_{i=1}^n
		|\ell_i'(\tilde f_i)|
		\,|f_{\phihat}(\bx_i)-f_{\phitilde}(\bx_i)|
		\le
		C(M,K,L)B^L(N+1)^{L+1}\delta
		\left(
		\frac{1}{n}\sum_{i=1}^n |\ell_i'(\tilde f_i)|
		\right).
	\end{align}
	To this end, we note the Bernstein's inequality for independent sub-exponential variables.
	\begin{lemma}\cite[Theorem 2.9.1]{vershynin2018high}
		\label{lmm:bernstein}
		Suppose that $Z_1,\dots,Z_n$ are subexponential random variables, with subexponential parameter $\sigma_{\max}$. Then we have 
		\begin{align}
			\PP\!\left(\frac{1}{n}\sum_{i=1}^n (Z_i-\EE[Z_i])\geq t\right)
			&\leq
			\exp\!\left(-c \cdot n \min\!\left\{\frac{t^2}{\sigma_{\max}^2},\;\frac{t}{\sigma_{\max}}\right\}\right).
		\end{align}
	\end{lemma}
	Since $y_i\sim N(0,\sigma^2 f^*(\bx_i)^2)$ and $f^*$ is bounded from above by $f_{\max}$, the above implies that with probability at least
	$1-\exp(-cn)$, we have
	\[
	\frac{1}{n}\sum_{i=1}^n y_i^2 = O(1).
	\]
	Hence, $\frac 1n\sum_{i=1}^n|\ell_i'(\tilde f_i)|$ is also $O(1)$ with a high probability.
	Substituting the above in \eqref{eq:m8} we get
	\begin{align}
		\label{eq:m27}
	|\ell(f_{\phihat}) - \ell(f_\phitilde)| \le C(M,K,L)B^L(N+1)^{L+1}\delta
	\leq \frac C{n^6},
	\end{align}
	with probability at least $1-\exp(-cn)$, for constants $c,C>0$, where we used $\delta = \frac 1{n^{(c_4+1)L+8}}$ from \eqref{eq:delta-value}.
	
	\subsection{Bounding the difference between the expected likelihoods from below}
	\label{app:expected-likelihood-diff}
	We first provide a lower bound to $\bar\ell(f_\phitilde)
	-\bar\ell(f^*)$. We use the following result, provided in \prettyref{app:technical}.
	\begin{lemma}\label{lmm:ineq1}
		$\inf_{x\in (0,\frac 1{2a}]}x -1 -\log x-a(x-1)^2\geq 0$ for $a<\frac 12$.
	\end{lemma}
	Using \prettyref{lmm:ineq1} with
	$
	a=\frac{1}{2M^4},
	$
	we get that whenever $0\le
	\frac{f^*(\bx)}{f_\phitilde(\bx)}
	\le M^2,$ which is satisfied in view of the assumptions on the range of $f^*,f_\phitilde$, we use \eqref{eq:calL} to get
	\begin{align}
		\label{eq:ineq-diff-lower}
		&\sum_{i=1}^n
		\left[
		\left(\frac{f^*(\bx_i)}{f_\phitilde(\bx_i)}\right)^2
		-1
		-\log\left(\frac{f^*(\bx_i)^2}{f_\phitilde(\bx_i)^2}\right)
		\right]
		\\
		&\ge
		a\sum_{i=1}^n
		\left(
		\frac{f^*(\bx_i)^2}{f_\phitilde(\bx_i)^2}-1
		\right)^2
		=
		a\sum_{i=1}^n
		\frac{(f^*(\bx_i)^2-f_\phitilde(\bx_i)^2)^2}{f_\phitilde(\bx_i)^4}
		\ge
		\frac{1}{2M^4}
		\sum_{i=1}^n
		\frac{(f^*(\bx_i)^2-f_\phitilde(\bx_i)^2)^2}
		{f_\phitilde(\bx_i)^2 f^*(\bx_i)^2}
		=
		\frac{\calL}2\frac{1}{M^8}.
		\nonumber
	\end{align}

	\subsection{Combining the results}
	\label{sec:combining}
	Combining the inequalities \eqref{eq:m7}, \eqref{eq:ineq-approx}, \eqref{eq:m27}, \eqref{eq:ineq-diff-lower}, in view of \eqref{eq:ineq-main}, we get with probability $1-\frac{1}{n^c}$
	\begin{align}
		\label{eq:ineq-combine}
		C_0\frac{\calL}{2n}
		&\leq
		\sqrt{\frac{A\log n}{n}}\sqrt{\frac{\calL}{n}}
		+ \frac{A\log n}{n}
		+
		\pth{n\over \log n}^{-{2\gamma^*\over 2\gamma^*+1}}
		\lesssim
		\sqrt{\frac{A\log n}{n}}\sqrt{\frac{\calL}{n}}
		+
		\pth{n\over \log n}^{-{2\gamma^*\over 2\gamma^*+1}},
	\end{align}
	for a constant $C_0=C_0(f_{\min},f_{\max},M)>0$, where the last inequality follows from using $N=\Theta\pth{(n/\log n)^{1\over 4\gamma^*+2}}$ and
	\begin{align}
		\label{eq:m28}
	{A\log n\over n}
	\approx\pth{N^2\log n\over n}
	\approx\pth{(n/\log n)^{-1+{2\over 4\gamma^*+2}}}
	\approx \pth{(n/\log n)^{-{2\gamma^*\over 2\gamma^*+1}}}.
	\end{align}
	Next, we note the fact that $ax^2\leq bx+c,a,b,c,x>0$ implies 
	$$x^2
	\leq \pth{b+\sqrt{b^2+4ac}\over 2a}^2
	\leq 
	2\pth{{b^2\over 4a^2}+{b^2+4ac\over 4a^2}}
	\leq 
	{b^2\over a^2} + {2c\over a}.$$ 
	In view of the above with $x={\calL\over n}$, \eqref{eq:ineq-combine} and \eqref{eq:m28} imply that for constants $C,\tilde C>0$, we have 
	\begin{align}
		\|f_\phitilde - f^*\|_n^2
		\leq 
		{\calL\over n} 4M^2
		\leq
		C\left({A\log n\over n} + \pth{n\over \log n}^{-{2\gamma^*\over 2\gamma^*+1}}\right)
		\leq 
		\tilde C \pth{n\over \log n}^{-{2\gamma^*\over 2\gamma^*+1}}.
	\end{align}
	In view of \prettyref{lmm:lipschitz} and $(N+1)^{L+1}\delta=o(\frac 1n)$, we use $(x+y)^2\leq 2(x^2+y^2)$ to get
	\begin{align*}
		\|f_\phihat - f^*\|_n^2
		&
		\leq
		2\|f_\phitilde - f^*\|_n^2
		+
		2\|f_\phitilde - f_\phihat\|_n^2
		\nonumber\\
		&
		\leq 
		2\tilde C \pth{n\over \log n}^{-{2\gamma^*\over 2\gamma^*+1}}
		+
		2\sup_{\bx\in [-K,K]^d}
		|f_\phitilde(\bx)-f_\phihat(\bx)|^2
		\notag\\
		&\stepa{\leq} 
		2\tilde C \pth{n\over \log n}^{-{2\gamma^*\over 2\gamma^*+1}}
		+
		\bar C(N+1)^{2(L+1)}\|\phihat-\phitilde\|_\infty^2
		{\leq} \breve C \pth{n/\log n}^{-{2\gamma^*\over 2\gamma^*+1}},
	\end{align*}
	 where $\bar C,\breve C>0$ are some constants. This completes the proof.

	\section{Proof of \prettyref{thm:sparse}}
	\label{app:theory-sparse}
	
	Throughout our proof we use $\tilde C,\breve C,\bar C$, with different subscripts, as positive constants depending on $f_{\min},f_{\max},M$, as necessary.
	
	In view of the exposition in \prettyref{sec:likelihood-diff}, we establish error guarantees for the neural network estimator $\hat f^\sparse= f_{\phihat}$ that minimizes the $\ell_1$-penalized negative log-likelihood $\ell(f)$ as in \eqref{eq:likelihood-reduced}, given by
	\begin{align}
		\label{eq:likelihood-reduced-penal}
	\hat f^\sparse(\bx)
	=f_{\phihat}(\Tr_M(\hat\Theta^\top \bx)),
	\quad
	f_\phihat,\Thetahat \in \argmin_{f_\phi\in \calG(L,d,1,N,M,B)\atop \Theta\in \reals^{d\times N}}
	\ell\pth{f_\phi|\sth{(y_i,\Tr_M\{\Theta^\top \bx_i\})}_{i=1}^n}
	+\lambda\sum_{i,j}\psi_\tau(\Theta_{i,j}).
	\end{align}
	
	Suppose that \(\calJ \subset \{1,\dots,n\}\) is the set of active indices in the support of \(f^*\). We done by $\bx_\calJ$ as the active coordinates of the feature vector $\bx$ and $\bx_{i,\calJ}$ denote the active components of the observed feature $\bx_i$ for $i=1,\dots,n$. Then we have
	\begin{align}
		\label{eq:true_sparse}
		f^*(\bx) = f^*(\bx_\calJ),
	\end{align}
	Consider as $\phi^*$ the parameter set of the DNN that is closest to the true data generating function $f^*$ over the feature vector space $\bx_\calJ$. In view of \prettyref{lmm:approx-smooth} we note that $\phistar$ can be chosen in such a way that
	\begin{align}
		\label{eq:approx-error-highdim}
		\sup_{\bx_\calJ\in [-K,K]^{|\calJ|}}\|f_{\phistar}(\bx_\calJ)-f^*(\bx_\calJ)\| < N^{-2\gamma^*},
	\end{align}
	where $\gamma^*$ is the degree adjusted smoothness defined in \prettyref{def:dim-adjusted-smoothness}.
	Let
	\[
	\calJ = \{l_1,\dots,l_{|\calJ|}\},
	\quad 
	\Thetastar_{ij} =
	\mathbf{1}\{ 1\leq i \le |\calJ|,\ j = l_i \}.
	\] 
	Given the minimizer $\hat \phi,\Thetahat$ of the optimization \eqref{eq:likelihood-reduced-penal}, let 
	\begin{align}
		\hat f^{\sparse}(\bx) = f^{\sparse}(\bx;\phihat,\hat\Theta)
		=f_\phihat(\Tr_M(\hat\Theta^\top \bx)).
	\end{align} 
	Using the parameters $\phistar,\Thetastar$, and noting that $\Thetastar$ has $|\calJ|$-many nonzero entries and $f_\phistar(\Tr_M(\Theta^{*\top} \bx))=f_\phistar(\bx),\sum_{i,j}\psi_\tau(\hat\Theta^*_{i,j})=|\calJ|$, we have
	\begin{align}
		\label{eq:m21}
		&\ell(\hat f^\sparse)
		+\lambda\sum_{i,j}\psi_\tau(\hat\Theta_{i,j})
		\le
		\ell\pth{f_\phistar}
		+
		\lambda |\calJ|
		.
	\end{align}

	Similar to \eqref{eq:ineq-main}, for any $\phitilde,\Thetatilde$, with $\tilde f^\sparse(\bx) =f^{\text{SPN}}(\bx;\phitilde,\Thetatilde)=f_{\phitilde}(\Tr_M(\Thetatilde^\top \bx))$, using \eqref{eq:likelihood-difference} we get
	\begin{align}
		\label{eq:idendity-main-highdim}
		\bar\ell(\tilde f^\sparse)
		-\bar\ell(f^*)
		&=
		\frac{
			\by^\top (\Sigmastar
			- \Sigmatilde^\sparse) \by
			-
			\E \big[ \by^\top (\Sigmastar
			- \Sigmatilde^\sparse) \by \big]
		}{n}
		+\ell(\tilde f^\sparse)- \ell(f^*),
	\end{align}
	where $\bar \ell(f)$ is as defined in \eqref{eq:cond_expectation} and we use the notations
	\[
	\Sigmatilde^\sparse
	=
	\diag\!\left(
	\{
	f_{\phitilde}(\Tr_M(\Thetatilde^{\top}\bx_i))^{-2}
	\}_{i=1}^{n}
	\right),
	\quad
	\Sigmastar
	=
	\diag\!\left(
	\{
	f^*(\bx_{i,\calJ})^{-2}\}_{i=1}^{n}
	\right).
	\]
	To bound $\ell(\tilde f^\sparse)$ using \eqref{eq:m21}, we use
	\begin{align*}
		\ell(\tilde f^\sparse)
		&=
		\ell(\hat f^\sparse)
		+\lambda\sum_{i,j}\psi_\tau(\hat \Theta_{i,j})
		+
		\ell(\tilde f^\sparse)
		-
		\ell(\hat f^\sparse)
		-\lambda\sum_{i,j}\psi_\tau(\hat \Theta_{i,j})
		\notag\\
		&\leq
		\ell\pth{f_\phistar}
		+\lambda|\calJ|
		+
		\ell(\tilde f^\sparse)
		-
		\ell(\hat f^\sparse)
		-\lambda\sum_{i,j}\psi_\tau(\hat \Theta_{i,j}).
	\end{align*}
	 
	In view of the above, we continue \eqref{eq:idendity-main-highdim} to get fo
	\begin{align}
		\label{eq:ineq-main-highdim}
		\bar\ell(\tilde f^\sparse)
		-\bar\ell(f^*)
		&\leq
		\frac{
			\by^\top (\Sigmastar
			- \Sigmatilde^\sparse) \by
			-
			\E \big[ \by^\top (\Sigmastar
			- \Sigmatilde^\sparse) \by \big]
		}{n}
		-\lambda\sum_{i,j}\psi_\tau(\Thetahat_{i,j})
		+\lambda|\calJ|
		\notag\\
		&\quad 
		+
		\ell(f_{\phi^*})-\ell(f^*)
		+\ell(\tilde f^\sparse)- \ell(\hat f^\sparse).
	\end{align}
	
	In view of the analysis in \prettyref{app:expected-likelihood-diff}, using that $f^\ast$ and $\tilde f^\sparse$ are bounded by $\frac 1M,M, f_{\min},f_{\max}$, we get that there is a constant $C_1=C_1(f_{\min},f_{\max},M)$ such that
	\begin{align}
		\label{eq:m15}
		\bar\ell(\tilde f^\sparse)
		-\bar\ell(f^*)
		\ge
		\frac{C_1}{n}
		\sum_{i=1}^{n}
		\left(
		f^*(\bx_i) - \tilde f^\sparse(\bx_i)
		\right)^2
		=C_1\|\hat f^\sparse - f^*\|_n^2.
	\end{align}
	Note that the above result is a universal bound and does not have any probabilistic component. 
	Using the last display, we continue \eqref{eq:ineq-main-highdim} to get that for any $\phitilde,\Thetatilde$ with $\tilde f^\sparse(\bx)=f_{\phitilde}(\Tr_M(\Thetatilde^\top \bx))$
	\begin{align}
		C_1\|\hat f^\sparse - f^*\|_n^2
		&\leq
		\frac{
			\by^\top (\Sigmastar
			- \Sigmatilde^\sparse) \by
			-
			\E \big[ \by^\top (\Sigmastar
			- \Sigmatilde^\sparse) \by \big]
		}{n}
		-\lambda\sum_{i,j}\psi_\tau(\Thetahat_{i,j})
		\notag\\
		&\quad 
		+\lambda|\calJ|
		+
		\ell(f_{\phi^*})-\ell(f^*)
		+\ell(\tilde f^\sparse)- \ell(\hat f^\sparse).
	\end{align} 
	Specializing the above display with $\phitilde=\phihat,\Thetatilde=\Thetahat$, we get
	\begin{align}
		\label{eq:ineq-main-hat}
		C_1\|\hat f^\sparse - f^*\|_n^2
		&\leq
		\frac{
			\by^\top (\Sigmastar
			- \Sigmahat^\sparse) \by
			-
			\E \big[ \by^\top (\Sigmastar
			- \Sigmahat^\sparse) \by \big]
		}{n}
		-\lambda\sum_{i,j}\psi_\tau(\Thetahat_{i,j})
		\notag\\
		&\quad
		+\lambda|\calJ| 
		+
		\ell(f_{\phi^*})-\ell(f^*).
	\end{align} 
	Next, we will use empirical process theory and a peeling device argument to bound
	\begin{align}
		\label{eq:term-estimation-1}
		T(\phihat,\Thetahat)
		=\frac{
			\by^\top (\Sigmastar
			- \Sigmahat^\sparse) \by
			-
			\E \big[ \by^\top (\Sigmastar
			- \Sigmahat^\sparse) \by \big]
		}{n}
		-\lambda\sum_{i,j}\psi_\tau(\Thetahat_{i,j}).
	\end{align}
	To this end, define the sets \(\calG_m, \mathcal{G}_{m,s} \) as follows:
	\begin{align}
		\mathcal{G}_m 
		&=
		\left\{
		f = f^{\sparse}(\cdot; g, \Theta) :
		g \in \mathcal{G}_{L,1,N,N,M,B},
		\ \Theta \in \mathbb{R}^{d\times N},
		\ \|\Theta\|_{\max} \le B
		\right\},
		\notag\\
		\mathcal{G}_{m,s}
		&=
		\left\{
		f = f^{\sparse}(\cdot; g, \Theta)
		\in \mathcal{G}_m :
		\sum_{i,j} \psi_\tau(\Theta_{ij}) \le s,
		\ \|\Theta\|_{\max} \le B
		\right\}.
	\end{align}
	We will extensively use the following result on the covering number bound.
 
	\begin{lemma}\cite[Lemma 7, without the diversified projection component]{fan2024factor}
		\label{lmm:g-infty-cover-number}
		There exists a universal constant $c_1$ such that for any $\delta > 2\tau KB^{L+1}N^{L+2}d$ and $N, L \ge 2$, 
		\begin{align*}
			\log |\mathcal{N}(\delta, \mathcal{G}_{m,s}, \norm{\cdot}_{\infty, [-K,K]^d})| \le c_1 \Bigg\{ &(N^2L) \left[L \log BN + \log \left(\frac{M\lor K}{\delta} \lor 1\right)\right] \\
			&~~~~~ + s \left[L \log (BN) + \log d + \log\left(\frac{KN}{\delta} \lor 1\right)\right] \Bigg\}.
		\end{align*}
	\end{lemma}
	Denote
	\begin{align}
		\label{eq:m12}
		\begin{gathered}
			\tilde A_\delta 
			=n(\nu_{n,\delta}+s\varrho_{n,\delta})
			\\
			\varrho_{n,\delta}
			=
			\frac{1}{n}
			\left[
			L\log BN
			+ \log d
			+ \log\!\left(\frac{KN}{\delta}\vee1\right)
			\right],
			\nu_{n,\delta}
			=
			\frac{N^2L}{n}
			\left[
			L\log BN
			+
			\log\!\left(\frac{MNK}{\delta}\vee1\right)
			\right].
		\end{gathered}
	\end{align}
	Note that by our choice of $\tau$ in \prettyref{thm:sparse}, we get that \prettyref{lmm:g-infty-cover-number} is satisfied for 
	\begin{align}
		\label{eq:delta-sparse}
		\delta = \Theta\pth{\frac 1{n^3}}.
	\end{align}
	
	In view of an analysis similar to \eqref{eq:ineq-estimation}, we get that for all $t>0$, with probability $1-e^{-t}$, 
	\begin{align}
		\label{eq:ineq-estimation-highdim}
		\abs{\by^\top(\Sigmastar-\Sigmatilde^\sparse)\by
			-
			\mathbb{E}\!\left[
			\by^\top(\Sigmastar-\Sigmatilde^\sparse)\by
			\right]}
		\leq C_2(f_{\min},f_{\max},M)\sth{
			\sqrt{(\tilde A_\delta +t)\tilde \calL^\sparse}
			+ \tilde A_\delta +t}
	\end{align}
	for all $\tilde f^\sparse$ in the a minimal $\delta$-covering set $\calN(\delta,\calG_{m,s},\norm{\cdot}_{\infty,[-K,K]^d})$ of $\calG_{m,s}$, and we use the notation
	\begin{align}
		\label{eq:L-spn}
		\tilde \calL^\sparse
		= \sum_{i=1}^{n}
		\left(
		f^*(\bx_i) - \tilde f^\sparse(\bx_i)
		\right)^2.
	\end{align}
	Next, we use Young's inequality \cite{young1912classes} to get
	\begin{align}
		2\sqrt{(\tilde A_\delta +t)\tilde \calL^\sparse}
		\leq 
		\epsilon\tilde \calL^\sparse
		+{\tilde A_\delta +t\over \epsilon},
		\quad\epsilon>0.
	\end{align}
	In view of the above and \eqref{eq:m12}, we can continue \eqref{eq:ineq-estimation-highdim} to get for any $\epsilon>0$, with probability $1-e^{-t}$,
	\begin{align}
		\label{eq:ineq-estimation-highdim-2}
		&\frac 1n\abs{\by^\top(\Sigmastar-\Sigmatilde^\sparse)\by
			-
			\mathbb{E}\!\left[
			\by^\top(\Sigmastar-\Sigmatilde^\sparse)\by
			\right]}
		\\
		&\leq C_2
		\sth{
			\frac \epsilon2 {\tilde\calL^\sparse\over n}
			+\pth{1+\frac 1{2\epsilon}}{\pth{\nu_{n,\delta}+s\varrho_{n,\delta} + \frac tn}}},
		\quad 
		\forall \ \tilde f^\sparse\in \calN(\delta,\calG_{m,s},\norm{\cdot}_{\infty,[-K,K]^d}),
		\notag
	\end{align}
	for a constant $C_2=C_2(f_{\min},f_{\max},M)>0$. Then, the above inequality leads to the following result.
	\begin{lemma}
		\label{lmm:chaining}
		Consider any $\gamma\geq \delta$, where $\delta$ is as in \eqref{eq:delta-sparse}. 
		Conditioned on fixed $\bx_1,\cdots, \bx_n$, define
	\begin{align*}
		Z_{s, \gamma} := 
		\sup_{f^\sparse\in \mathcal{G}_{m,s}, \| f^\sparse- f^*\|_n \le \gamma} \frac 1n\abs{\by^\top(\Sigmastar-\Sigma^\sparse)\by
			-
			\mathbb{E}\!\left[
			\by^\top(\Sigmastar-\Sigma^\sparse)\by
			\right]}.
	\end{align*} 
	where $f^\sparse$ is a representative DNN from $\calG_{m,s}$ with parameters $\phi,\Theta$, and 
	\begin{align*}
		\Sigma
		=
		\diag\!\left(
		\{
		f_\phi(\Tr_M(\Theta^{\top}\bx_i))^{-2}\}_{i=1}^{n}
		\right).
	\end{align*} 
	Then there exists a constant $C_3>0$, such that for any $\epsilon>0$, we have
	\begin{align}
		\label{eq:claim-step4.1}
		\begin{gathered}
			\mathbb{P}\left(\mathcal{B}_{t,\epsilon}(s,\gamma)\right) \ge 1-2e^{-t}, ~~\text{where}~~ \mathcal{B}_{t,\epsilon}(s,\gamma) = \left\{Z_{s, \delta} 
			\le 
			C_3 \left[\epsilon\gamma^2  + (1+\frac 1{2\epsilon})(\nu_n+s\varrho_n+\frac tn)\right]\right\},\\
			\nu_n = {L^2N^2\log(BNn)\over n},
			\quad
			\varrho_n={\log(Ndn)+L\log(BN)\over n}.
		\end{gathered}
	\end{align}
	\end{lemma}

	The rest of the proof uses a standard peeling device argument \cite{geer2000empirical} to bound \eqref{eq:term-estimation-1}. We will peel the function class $\mathcal{G}_m$ with respect to $\sum_{i,j}\psi_\tau(\Theta_{i,j})$. More specifically, define
	\begin{align*}
		\mathcal{G}_{m,k,\ell} = \left\{\tilde f^\sparse \in \mathcal{G}_m, \alpha_{k-1} \le \sum_{i,j} \psi_\tau(\Theta_{i,j}) \le \alpha_k, \alpha_{\ell-1} \sqrt{\nu_n} < \|\tilde f^\sparse - f^*\|_n \le \alpha_{\ell} \sqrt{\nu_n} \right\}
	\end{align*} with $\alpha_i = 2^i$ for $i\ge 0$ and $\alpha_{i} = 0$ for $i=-1$ and note that
	$
		\mathcal{G}_m = \bigcup_{k=0}^{\lceil \log_2 (dN)\rceil} \bigcup_{\ell=0}^{\lceil \log_2 (2M/\sqrt{\nu_n})\rceil} \mathcal{G}_{m,k,\ell}.
	$
	Recall $\tilde \calL^\sparse$ from \eqref{eq:L-spn}. We first establish a lower bound on the probability of the event
	\begin{align*}
		\mathcal{B}_{t,\epsilon}(k,\ell) = \left\{\forall \tilde f^\sparse\in \mathcal{G}_{m,k,\ell}, T(\phitilde,\Thetatilde) \le C_4
		\sth{{\epsilon}\|\tilde f^\sparse - f^*\|_n^2
			+\frac 1{\epsilon}\pth{\nu_n+\varrho_n+\frac tn}}\right\},
	\end{align*} 
	some constant $C_4>0$. Note that $\tilde f^\sparse\in \mathcal{G}_{m,k,\ell}$ implies $\alpha_{k-1}\leq\sum_{i,j} \psi_\tau({\Theta}_{i,j}) \leq \alpha_{k}$ by the definition of the set $\mathcal{G}_{m,k,\ell}$. Then, on the event $\calB_{t,\epsilon}(\alpha_k,\alpha_\ell)$, for all the $\tilde f^\sparse\in \mathcal{G}_{m,k,\ell}$, we have 
	\begin{align*}
		T(\phitilde,\Thetatilde)
		&\leq 
		\frac 1n\abs{\by^\top(\Sigmastar-\Sigmatilde^\sparse)\by
			-
			\mathbb{E}\!\left[
			\by^\top(\Sigmastar-\Sigmatilde^\sparse)\by
			\right]} - \lambda \alpha_{k-1} \\
		&\overset{(a)}{\le}
		C_3 \left[\epsilon \alpha_\ell^2 \nu_n  + \pth{1+\frac 1{2\epsilon}}\pth{\nu_n+\alpha_k\varrho_n+\frac tn}\right] - \lambda \alpha_{k-1} \\
		&\stepb{\leq} C_3\qth{{\epsilon}(4 \alpha_{\ell-1}^2 + 1) \nu_n + \frac{2}{\epsilon} (\nu_n + \varrho_n  + \frac{t}{n})} + \left(2{C_3}\pth{1+\frac 1{2\epsilon}}\varrho_n -\lambda\right) \alpha_{k-1} \\
		&\stepc{\leq}
		 C_3\qth{4\epsilon\alpha_{\ell-1}^2\nu_n + \frac{3}{\epsilon} (\nu_n + \varrho_n  + \frac{t}{n})}
		\\
		&\stepd{\leq} 
		C_4\qth{\epsilon \|\tilde f^\sparse - f^*\|_n^2 + \frac{1}{\epsilon} (\nu_n + \varrho_n  + \frac{t}{n})},
	\end{align*} 
	where (a) follows from the definition of the event $\mathcal{B}_{t}(\alpha_k)$, and (b) uses the relationship $\alpha_{i} \le 2\alpha_{i-1} + 1$ between $\alpha_{i-1}$ and $\alpha_i$, (c) follows as we choose $\lambda>{\tilde C_3\over \epsilon}\varrho_n$ for a large enough constant $\tilde C_3$ and $\epsilon$ is of constant order, and(d) follows for a large constant $C_4>0$ as $\|\tilde f^\sparse - f^*\|_n\geq \alpha_{\ell-1}\sqrt \nu_n$ on the event$\calB_{t,\epsilon}(\alpha_k,\alpha_\ell)$. Hence we have
	\begin{align*}
		\mathbb{P}\left[\big(\mathcal{B}_{t,\epsilon}(k,\ell)\big)^c\right] \le \mathbb{P}\left[ \big(\mathcal{B}_{t,\epsilon}(\alpha_k,\alpha_\ell)\big)^c\right]\le 2e^{-t}. 
	\end{align*} 
	provided $\lambda \ge {\tilde C_3\over \epsilon}\varrho_n$. Then, it follows from a union bound that
	\begin{align*}
		&\mathbb{P}\left[\exists \tilde f^\sparse\in \mathcal{G}_{m}, \text{ such that } 
		T(\phitilde,\Thetatilde) > C_4\qth{{\epsilon}\|\tilde f^\sparse - f^*\|_n^2
		+\frac 1{\epsilon}\pth{\nu_n+\varrho_n+\frac tn}} \right]\\
		&
		\leq \sum_{k=0}^{\lceil \log_2 (pN)\rceil} \sum_{\ell=0}^{\lceil \log_2 (2M/\sqrt{\nu_n})\rceil} \mathbb{P}\left[\big(\mathcal{B}_{t,\epsilon}(k,\ell)\big)^c\right] \le C_5 \log(dN) \log (2Mn) e^{-t},  
	\end{align*} 
	for any $t>0$. Replacing $t$ by $t+C_6\log \left( \log(dN) \log (Mn)\right)$ for some large constant $C_6>0$ and noting that $\log \left(\log(dN) \log (Mn)\right) \lesssim n\varrho_n$, we have
	\begin{align*}
		&\mathbb{P}\left[
		T(\phitilde,\Thetatilde) < C_7\qth{{\epsilon}\|\tilde f^\sparse - f^*\|_n^2
			+\frac 1{\epsilon}\pth{\nu_n+\varrho_n+\frac tn}} \text{ for all } \tilde f^\sparse\in \mathcal{G}_{m}\right]
		\geq 1-e^{-t}.  
	\end{align*}
	This completes our proof of bounding $T(\phitilde,\Thetatilde)$ in \eqref{eq:term-estimation-1}, for all $\tilde f^\sparse\in \mathcal{G}_{m}$. Hence, choosing $\phitilde,\Thetatilde=\phihat,\Thetahat$, we get that with probability $1-e^{-t}$ we have
	\begin{align}
		T(\phihat,\Thetahat) < C_7\qth{{\epsilon}\|\hat f^\sparse - f^*\|_n^2
			+\frac 1{\epsilon}\pth{\nu_n+\varrho_n+\frac tn}}.
	\end{align}
	Combining the above bound with \eqref{eq:ineq-main-hat}, we get with probability $1-e^{-t}$, for any $\epsilon>0$
	\begin{align*}
		C_1 \|\hat f^\sparse - f^*\|_n^2
		\leq 
		C_7\qth{{\epsilon}\|\hat f^\sparse - f^*\|_n^2
			+\frac 1{\epsilon}\pth{\nu_n+\varrho_n+\frac tn}}
		+\lambda|\calJ|
		+
		\ell(f_{\phi^*})-\ell(f^*).
	\end{align*}
	
	Hence, we can choose $\epsilon$ to be a small enough positive constant (possibly depending on $f_{\min},f_{\max}$) such that the above expression simplifies to
	\begin{align}
		\label{eq:m16}
		C_8 \|\hat f^\sparse - f^*\|_n^2
		\leq 
		C_9{\pth{\nu_{n}+\varrho_{n} + \frac tn}}
		+\lambda|\calJ|
		+
		\ell(f_{\phi^*})-\ell(f^*).
	\end{align}	
	In view of \eqref{eq:ineq-approx} and the definition of $\nu_n$ in \prettyref{lmm:chaining}, with $t = c_5{\log n}$ the desired upper bound follows, with probability $1-\frac 1{n^{c_5}}$,
	\begin{align}
		\|\hat f^\sparse - f^*\|_n^2
		\leq C_{10}{N^2\log n\over n}
		\leq 
		\frac {C_{11}}n\pth{n\over \log n}^{1\over 2\gamma^*+1}\log n
		\leq
		C_{11} \pth{n\over \log n}^{-{2\gamma^*\over 2\gamma^*+1}}.
	\end{align}
	where $C_{7},C_8,\dots>0$ are appropriate constants as required.
	
	\section{Proof of \prettyref{prop:population-empirical-connection}}
	\label{app:proof-population-loss}

	Our proof uses the result \cite[Theorem 19.3]{gyorfi2002distribution}. We specialize the result below to use in our setup. 
	Let $\alpha={c\log n\over n}$ for some large constant $c>0$. Suppose that the following conditions hold
	\begin{itemize}
		\item 
		There is a function class $\calF$ such that for all $f\in \calF$ we have $\|f-f^*\|_\infty^2 \le K_1, \mathbb{E} [(f(\bx)-f^*(\bx))^4] \le K_2 \mathbb{E} [(f(\bx)-f^*(\bx))^2]$ for some constants $K_1,K_2\geq 1$.
		\item  Let $\calN(\delta,\calG,\norm{\cdot}_n)$ be the $\delta$-cover of $\calG$ with respect to the $\norm{\cdot}_n$ norm. Then, there exists a constant $\tilde c>0$, such that for any constant $\bar c>0$ and for all $\delta>{\alpha\over 8}$ we have
		$$
		\log \mathcal{N}\left(u, \left\{f\in \calF, \|(f-f^*)^2\|_n^2 \le 16\delta \right\}, \norm{\cdot}_n\right)
		\leq 
		n^{\frac 12-\tilde c},
		\quad
		u\in \qth{\bar c\delta,\sqrt \delta}.$$
	\end{itemize}
	Then, there is a constant $\tilde C>0$, such that for all $n$ large enough, we have for all $\eta>0$
	\begin{align*}
		\mathbb{P}\left[\sup_{f\in {\calF}} \frac{|\|f-f^*\|_{2,\calQ}^2 - \|f-f^*\|_n^2|}{\alpha + \|f-f^*\|_{2,\calQ}^2} > \eta \right] 
		\le \exp\left(-\tilde C{n\alpha \eta^2(1-\eta)}\right).
	\end{align*}
	By choosing $\eta=\frac 12$, the above inequality implies that with a high probability, $\|f-f^*\|_{2,\calQ}^2$ is bounded from above by constant multiples of $\|f-f^*\|_n^2+{c\log n\over n}$, which guarantees the required nonparametric error rate on $\|f-f^*\|_{2,\calQ}^2$.
	
	We conclude our proof by verifying the conditions with $\calF = \calG(L,d,1,N,M,B)$. The first condition can be verified as all $f\in \calG(L,d,1,N,M,B)$ are truncated at the level $M$ to ensure boundedness and we assume that the true function $f^*$ is bounded. According to \prettyref{asmp:model-and-hyperparameters}, note that $f^*$ and $f\in\mathcal F$ satisfy
	\[
	\|f^*\|_\infty,\|f\|_\infty\le M.
	\]
	
	Next, we verify the second condition.  we have
	\[
	|f(\bx)-f^*(\bx)|
	\le 2M,
	\quad 
	(f(\bx)-f^*(\bx))^4
	\le
	4M^2(f(\bx)-f^*(\bx))^2,
	\quad
	\bx\in \reals^d.
	\]
	Averaging over the features $\bx_1,\dots,\bx_n$ gives
	$
	\|(f-f^*)^2\|_n^2
	\le
	4M^2\|f-f^*\|_n^2.
	$
	Hence, we get
	\[
	\mathcal F_\delta
	\eqdef
	\left\{f\in \calF, \|(f-f^*)^2\|_n^2 \le 16\delta \right\}
	\subset
	\left\{
	f\in\mathcal F:
	\|f-f^*\|_n
	\le
	\frac{2\sqrt{\delta}}{M}
	\right\},
	\]
	and consequently
	\[
	\mathcal N
	\left(
	u,
	\mathcal F_\delta,
	\norm{\cdot}_n
	\right)
	\le
	\mathcal N
	\left(
	u,
	\mathcal F,
	\norm{\cdot}_n
	\right).
	\]
	Next, since
	$
	\|f-g\|_n
	\le
	\|f-g\|_\infty,
	$
	we have that every $L_\infty$-cover is also an empirical $L_2$-cover. Therefore,
	\begin{align}
		\label{eq:m19}
		\mathcal N_2
		(u,\mathcal F,\norm{\cdot}_n)
		\le
		\calN(u,\calF,\norm{\cdot}_\infty).
	\end{align}
	where $\calN(u,\calF,\norm{\cdot}_\infty)$ is the minimal $u$-cover of the set $\calF$ in the supremum norm. The number of parameters in any network of $\calG(L,d,1,N,M,B)$ is $A=N^2L+N(L+d+2)$ (see \eqref{eq:total-parameters}), for finite $d,L$ we have
	$$
	\log|\calN(u,\calF,\norm{\cdot}_\infty)|
	\stepa{\leq}
	C_0 A\log\pth{n} 
	\leq C_1 N^2 \log n
	\stepb{\leq} 
	C_3\log n(n/\log n)^{1\over 2\gamma^*+1},
	$$
	for constants $C_0,C_1,C_3>0$, where (a) follows from the covering number bound \eqref{eq:covering-bound} and (b) follows from \prettyref{asmp:model-and-hyperparameters}\ref{pt:hyperparam}. As we have $\gamma^*>\frac 12$, we get that there exists $\tilde c>0$ that satisfies $\log|\calN(u,\calF,\norm{\cdot}_\infty)|\leq n^{1/2-\tilde c}$. In view of \eqref{eq:m19}, the above implies
	\[
	\log
	\mathcal N_2
	\left(
	u,
	\mathcal F_\delta,
	\norm{\cdot}_n
	\right)
	\le
	n^{1/2-\tilde c},
	\]
	for all sufficiently large $n$. This completes our proof.

	\section{Proof of \prettyref{thm:speckle-lb}}
	
	\label{sec:proof-lower-bound}
	
	\subsection{Auxiliary results}
	We first present a few results that we use throughout our proof of the lower bound.
	The next two lemmas are two basic concentration results, one for simple quadratic functions of Gaussian random vectors, and the second one for the concentration of sum of subexponential random variables. 
	
	\begin{lemma}\label{lem:lM}\cite{laurentmassart} 
		Let $Z_1, Z_2, \ldots, Z_n$ be iid random variables distributed as $N(0,1)$ and $a \in \mathbb{R}^n_{\geq 0}$. Then for all $t\geq 0$,
		$$\mathbb{P}\left(\sum_i a_i Z^2_i- \sum_i a_i \geq 2\|a\|_2\sqrt{t} + 2\|a\|_\infty t \right) \leq e^{-t},
		\quad
		\mathbb{P}\left(\sum_i a_i Z^2_i- \sum_i a_i \leq -2\|a\|_2\sqrt{t} \right) \leq e^{-t}.$$
	\end{lemma}
	
	\begin{lemma}\label{berstein's inequality for sub-exponential distribution}[Berstein's inequality for sub-exponential distributions. Theorem 2.9.1, \cite{vershynin2018high}]
		Let $Z_1,...,Z_n$ be independent mean zero, sub-exponential random variables. Then, for $t>0$, we have
		\begin{equation}
			\mathbb{P}\qth{\abs{\sum_{j=1}^nZ_j}\ge t}\le \exp\pth{-c\min\pth{\frac{t^2}{\sum_{j=1}^n \|Z_j\|_{\rm subexp}^2},\frac{t}{\max_{1
							\le j\le n}\|Z_j\|_{\rm subexp}}}}
		\end{equation}
		where $c>0$ is an absolute constant.
	\end{lemma}
	
	The next lemma is a version of the well-known Berry-Esseen theorem. 
	
	\begin{lemma}\label{lem:BE}(non-iid Berry-Esseen theorem \cite{Ber-Ess})
		Let $Z_1, Z_2, \ldots, Z_n$ be independent, zero mean random variables with finite second moment. Suppose that there exists a fixed number $\kappa$ such that  
		$$\max_{1\leq i \leq n}\frac{\mathbb{E}|Z_i|^3}{\mathbb{E}(Z_i)^2} \leq \kappa.$$
		If $F, \Phi$ denote the CDFs of $\frac{\sum_i Z_{i}}{\sqrt{\sum_i \mathbb{E} Z_{i}^{2}}}$ and $N(0,1)$ respectively, then there exists a constant $C_0$, such that
		$$\sup _{x}\left|F(x)-\Phi(x)\right|\leq C_0 \frac{\kappa}{\sqrt{\sum_i \mathbb{E} Z_{i}^{2}}}.$$
	\end{lemma}

	Our last lemma will be used in the proofs of the lower bounds.

	\begin{lemma}\label{lem:GV}(Gilbert-Varshamov bound \cite{Gilbert},\cite{Varshamov}, \cite{Mainbook}) 
		Let $\Omega=\{0,1\}^m$, where $m \geq 8$. Then, there exists a hypercube  $\Omega^{\prime}:=\left\{\omega^1, \ldots, \omega^M\right\} \subseteq \Omega$ such that
		$M \geq 2^{m / 8}$, each $\omega_i$ has at least $m/16$ nonzero entries, and we have
		$\rho\left(\omega^j, \omega^k\right) \geq m / 16$ for each $j \neq k$.
	\end{lemma}
	
	\subsection{Proof details}

	Let $\beta^*, d^*$ be as in \eqref{eq:gamma*} and $\gamma^*=\beta^*/d^*$. When $d^* \le |\calJ|$ we have the class of $(\beta^*,C)$ functions is a subset of $\mathcal{H}(|\calJ|,l,\mathcal{P})$.
	Hence, as $f^*(\bx)=f^*(\bx_{\calJ})$ is a $d^*$ dimensional function, we get for any $a>0$
	\begin{align*}
		&\inf_{\hat f} \sup_{f^* \in \mathcal{H}(|\mathcal{J}|,l,\mathcal{P})} \PP\qth{\|\hat f-f^*\|_n^2 \ge a}
		\ge \inf_{\hat f} \sup_{f^*:\reals^{d^*}\mapsto \reals^+,f^*\text{ is $(\beta^*,C)$ smooth}} \PP\qth{\|\hat f-f^*\|_n^2 \ge
			a},
	\end{align*}
	In view of the above, it is sufficient to establish
	\begin{align*}
		\inf_{\hat f} \sup_{f^*:\reals^{d^*}\mapsto \reals^+,f^*\text{ is $(\beta^*,C)$ smooth}} \PP\qth{\|\hat f-f^*\|_n^2 \ge c_1 n^{-\frac{2\gamma^*}{2\gamma^*+1}}}>p^*.
	\end{align*} 
	In the remaining part, our strategy extends the technique for proving \cite[Theorem 2.1]{malekian2025speckle}, which treated the special case of $d^*=1$, to the general dimension $d^*$. 
	First we set
	
	\begin{align}\label{eq:Mn}
		M_n = \left\lfloor (C^2 n)^{\frac{1}{2\beta^*+d^*}} \right\rfloor,
		\quad
		\delta_n=M_n^{-1}.
	\end{align}
	
	Partition $[0,1]^{d^*}$ by $M_n^{d^*}$ cubes $\{A_{n,j}\}$ of side length $1/M_n$ and with centers $\{\ba_{n,j}\}$.
	Choose a function $\bar g : \mathbb{R}^{d^*} \to \mathbb{R}$ such that the support of $\bar g$ is a subset of
	$[-\tfrac12,\tfrac12]^{d^*}$, $\int \bar g^2(\bx)\,d\bx>0$, and $\bar g$ is a $d^*$-dimensional $(\beta^*,2^{s-1})$ smooth function, where $\beta^*=r+s$, with $r\geq 0$ an integer and $s\in (0,1)$ a proper fraction. Define
	\[
	g(\bx)=C\,\bar g(\bx).
	\]
	Then we have the following.
	\begin{enumerate}
		\item the support of $g$ is a subset of $[-\tfrac12,\tfrac12]^{d^*}$;
		\item $\int g^2(\bx)\,d\bx=C^2\int \bar g^2(\bx)\,d\bx$ and $\int \bar g^2(\bx)\,d\bx>0$;
		\item $g$ is a $(\beta^*,C2^{s-1})$ smooth function.
	\end{enumerate}
	We specify the data generating regression function class based on the above construction. Define the regression function $\nu^{(\bomega_l)}(x)$, indexed by the vector $\bomega_l=(\omega_{l,1},\dots,\omega_{l,M_n^{d^*}})\in \{0,1\}^{M_n^{d^*}}$, as 
	$$
	\nu^{(\bomega_l)}(\bx)=1-\sum_{j=1}^{M_n^{d^*}} \omega_{l,j} g_{n,j}(\bx),
	\quad 
	g_{n,j}(\bx)=M_n^{-\beta^*} g(M_n(\bx-\ba_{n,j})).
	$$
	We first check that $\nu^{(\bomega_l)}$ is $(\beta^*,C)$-smooth. To show the above, set $\alpha=(\alpha_1,\dots,\alpha_{d^*})$, $\alpha_i\in\mathbb{N}_0$ with $\sum_{j=1}^{d^*} \alpha_j=r$, where $\beta^*=r+s$ as defined in \prettyref{def:beta-C-smooth}, and
	\[
	D^\alpha = \frac{\partial^r}{\partial x_1^{\alpha_1}\cdots \partial x_{d^*}^{\alpha_{d^*}}}.
	\]
	If $\bx,\bz\in A_{n,i}$ then we have
	\begin{align}
		&~|D^\alpha \nu^{(\bomega_l)}(\bx)-D^\alpha \nu^{(\bomega_l)}(\bz)|
		=|\omega_{l,i}|\cdot|D^\alpha g_{n,i}(\bx)-D^\alpha g_{n,i}(\bz)|
		\nonumber\\
		&\le C2^{s-1} M_n^{-s} M_n^r
		\|M_n(\bx-\ba_{n,i})-M_n(\bz-\ba_{n,i})\|^s
		\le C2^{s-1}\|\bx-\bz\|^s
		\le C\|\bx-\bz\|^s .
	\end{align}
	
	Now assume that $\bx\in A_{n,i}$ and $\bz\in A_{n,j}$ for $i\neq j$.
	Choose $\bar \bx,\bar \bz$ on the line between $\bx$ and $\bz$ such that
	$\bar \bx$ is on the boundary of $A_{n,i}$ and $\bar \bz$ on the boundary of $A_{n,j}$.
	Then we have
	\begin{align}
		&~|D^\alpha \nu^{(\bomega_l)}(\bx)-D^\alpha \nu^{(\bomega_l)}(\bz)|
		\le |\omega_{l,i}||D^\alpha g_{n,i}(\bx)-D^\alpha g_{n,i}(\bar \bx)|
		+|\omega_{l,j}||D^\alpha g_{n,j}(\bz)-D^\alpha g_{n,j}(\bar \bz)|
		\nonumber\\
		&
		\le C2^{s-1}(\|\bx-\bar \bx\|^s+\|\bz-\bar \bz\|^s)
		\stepa{\leq} C2^s\left(\frac{\|\bx-\bar \bx\|}{2}+\frac{\|\bz-\bar \bz\|}{2}\right)^s
		\le C\|\bx-\bz\|^s
	\end{align}
	where (a) follows by Jensen's inequality. To this end we use the next lemma which is a restatement of Theorem 2.5.3 in \cite{korostelev2012minimax}, clarifies what properties we expect the functions to satisfy, and what lower bound we can obtain for $\max_{0\leq j\leq M}\mathbb{P}_{\nu_j}(\frac 1n\sum_{i=1}^n(\hat\nu(\bx_i)-\nu_j(\bx_i))^2\geq t)$. Define the likelihood ratio $\Lambda(\nu', \nu'')$ under two different data generating functions $\nu', \nu''$ as
	\begin{align}\label{eq:likelihood-ratio}
		\Lambda\left(\nu^{\prime}, \nu^{\prime \prime}\right):= \Lambda\left(\nu^{\prime}, \nu^{\prime \prime};y_1, \ldots, y_n\right)= \mathbb{P}_{\nu^{\prime}}(y_1, \ldots, y_n)/\mathbb{P}_{\nu^{\prime \prime}}(y_1, \ldots, y_n),
	\end{align}
	where $\mathbb{P}_{\nu}(y_1, \cdots, y_n)$ denotes the probability density function of $y_1, y_2, \cdots, y_n$ under the data generating functions $\nu$.

	\begin{lemma}\label{lem:lowerbound-korostelev}
		Suppose $\nu_0 , \nu_1, \ldots, \nu_\tM$ are real-valued functions such that the following conditions hold:
		
		(i) For all $0 \leq j \neq k \leq \tM$, $\dist(\nu_j,\nu_k) \geq 2 s_n >0$, where $\dist$ is a distance on the class of $(\beta^*,C)$-smooth functions. 
		
		(ii) For all $j \in \{1,2,\ldots, \tM\}$, there exists $\lambda < 1$ such that $\lambda_j < \lambda$ and we have $$\mathbb{P}_{\nu_j}(\Lambda(\nu_0, \nu_j) > \tM^{-\lambda_j}) > p^* ,$$ where $p^*$ is independent of $n$, $j$, and $\Lambda (\nu_0, \nu_j)$ is defined in \eqref{eq:likelihood-ratio}. Then, for every estimator $\hat{\nu}_n$, we have 
		$$\max_{0\leq j\leq \tM}\mathbb{P}_{\nu_j}(\dist(\hat{\nu}_n, \nu_j)\geq s_n)\geq p^*/2.$$
	\end{lemma}

	We will use the above result with the distance functional $\dist$ given by the $\norm{\cdot}_{n}^2$ metric
	\begin{align}
		\dist(\nu,\mu)
		=\frac 1n\sum_{i=1}^n(\nu(\bx_i)-\mu(\bx_i))^2.
	\end{align}
	
	For a simplicity of notations, define $m=M_n^{d^*}$.  Now define $\nu_0\equiv 1$, choose the number of hypothesis $\tM$ as $2^{\frac m8}$,
	and let $\nu_1,\ldots,\nu_\tM$ be elements of
	$\{\nu^{(\bomega)}:\bomega\in\{0,1\}^m\}$ with
	$\nu_l=\nu^{\bw_l}$ for
	$\{\bomega_1,\dots,\bomega_\tM\}\subseteq \{0,1\}^{m}$. Then, \prettyref{lem:GV} guarantees that $\bomega_1,\dots,\bomega_\tM$ can be chosen such that
	$\rho(\bomega_j,\bomega_k)\geq m/16$ for all $j\neq k$ and each
	$\bomega_l$ has at least $m/16$ nonzero entries.
	
	\textbf{Separation of the hypotheses.} Consider two functions $\nu_j,\nu_k$ with $j\neq k$. We have the representation
	\begin{align}
		\nu_j(\bx) = \nu^{(\bomega_j)} = 1-\sum_{i=1}^{m}\omega_{j,i}g_{n,i}(\bx),
		\quad
		\nu_k(\bx) = \nu^{(\bomega_k)} = 1-\sum_{i=1}^{m}\omega_{k,i}g_{n,i}(\bx).
	\end{align}
	As the functions $\sth{g_{n,i}}_{i=1}^m$ have disjoint supports and the design points
	$\bx_1,\dots,\bx_n$ are fixed, we get
	\begin{align}\label{eq:emp:sep}
		\dist(\nu_j,\nu_k)
		&= \frac{1}{n}\sum_{i=1}^n\!\bigl(\nu_j(\bx_i)-\nu_k(\bx_i)\bigr)^2
		\nonumber\\
		&=
		\frac 1n \sum_{i=1}^n
		\pth{\sum_{l=1}^m(\omega_{k,l}-\omega_{j,l})g_{n,l}(\bx_i)}^2
		\geq 
		\sum_{\{l\,:\,\bomega_{j,l}\neq \bomega_{k,l}\}}
		\frac{1}{n}\sum_{i=1}^n g_{n,l}^2(\bx_i).
	\end{align}
	Recall $\delta_n,M_n$ as defined in \eqref{eq:Mn}. For each fixed $j$, the Riemann sum approximation gives
	\[
	\frac{1}{n}\sum_{i=1}^n g_{n,l}^2(\bx_i)
	= \delta_n^{2\beta^*}\,\frac{1}{n}\sum_{i=1}^n
	g^2\!\!\left(\tfrac{\bx_i-\ba_{n,l}}{\delta_n}\right)
	= \delta_n^{2\beta^*+d^*}\bigl(\|g\|_2^2+o(1)\bigr),
	\]
	where the $o(1)$ term is the Riemann sum error, which vanishes by \prettyref{asmp:riemann-approx-revised} as
	$n\delta_n^{d^*}\to\infty$.  Since $\rho(\bomega_j,\bomega_k)\geq m/16$, at
	least $m/16$ terms contribute to the sum over $l$ in \eqref{eq:emp:sep}. This implies
	\begin{equation}\label{eq:emp:sep:final}
		\dist(\nu_j,\nu_k)
		\;\geq\; \frac{m}{16}\,\delta_n^{2\beta^*+d^*}\bigl(\|g\|_2^2+o(1)\bigr)
		\;=\; \frac{\delta_n^{2\beta^*}\|g\|_2^2}{16}\bigl(1+o(1)\bigr)
		\;=:\; 2s_n,
	\end{equation}
	where we used $m\delta_n^{d^*}=1$.
	
	\textbf{Likelihood ratio.}
	Under the fixed design model \eqref{eq:model}, we have
	$y_i\sim N(0,\nu^2(\bx_i)+\sigma_\tau^2)$ independently.
	Writing $\prod_i=\prod_{i=1}^n$ and $\sum_i=\sum_{i=1}^n$,
	\[
	\begin{aligned}
		\Lambda(\nu_0,\nu_l)
		&=\prod_i\sqrt{\frac{\sigma_\tau^2+\nu_l^2(\bx_i)}{1+\sigma_\tau^2}}
		\exp\!\left[\sum_i\frac{y_i^2}{2(\sigma_\tau^2+\nu_l^2(\bx_i))}
		-\sum_i\frac{y_i^2}{2(1+\sigma_\tau^2)}\right]\\
		&=\exp\!\left[\sum_i\frac{y_i^2}{2(1+\sigma_\tau^2)}
		\left(\frac{\sigma_\tau^2+1}{\sigma_\tau^2+\nu_l^2(\bx_i)}-1\right)
		+\frac{1}{2}\sum_i\log\frac{\sigma_\tau^2+\nu_l^2(\bx_i)}{\sigma_\tau^2+1}
		\right].
	\end{aligned}
	\]
	
	To establish that the above likelihood ratio satisfies the requirements of \prettyref{lem:lowerbound-korostelev}, we use the following result.

	\begin{lemma}\label{lem:likelihoodratio}
		There exists $\lambda<1$ such that for all $l=1,\dots,\tM$,
		$\mathbb{P}_{\nu_l}(\Lambda(\nu_0,\nu_l)>\tM^{-\lambda_l})>p^*$,
		where $\lambda_l<\lambda$.
	\end{lemma}
	A proof of the above result is provided below.
	Combining \prettyref{lem:lowerbound-korostelev} with \eqref{eq:emp:sep:final}, we get,
	\[
	\max_{0\le j\le \tM}
	\mathbb{P}_{\nu_j}\!\left(
	\frac{1}{n}\sum_{i=1}^n(\hat\nu(\bx_i)-\nu_j(\bx_i))^2\geq s_n
	\right)\geq p^*/2.
	\]
	Hence, using the value of $\delta_n$ from \eqref{eq:Mn} and $s_n= \Omega(\delta_n^{2\beta^*})$ from \eqref{eq:emp:sep:final}, we get the desired lower bound of the order $\Omega\!\left(n^{-\frac{2\beta^*}{2\beta^*+d^*}}\right)$.
	This completes the proof of the lower bound.
	
	Finally, we end the section with a proof of \prettyref{lem:likelihoodratio}. The proof closely follows the proof of the special case of $d=1$ provided in \cite[Lemma 3.12]{malekian2025speckle}, extended here to the multidimensional scenario. We provide the details here for completeness.

	\begin{proof}[Proof of \prettyref{lem:likelihoodratio}]
		Our goal is to prove that for
		$y_i:=\nu_l(\bx_i)\,\xi_i+\tau_i$, $i=1,\ldots,n$,
		\begin{align}\label{eq:m6}
			\begin{split}
				&\mathbb{P}_{\nu_l}\!\left(
				\sum_i\frac{y_i^2}{1+\sigma_\tau^2}
				\left(\frac{\sigma_\tau^2+1}{\sigma_\tau^2+\nu_l^2(\bx_i)}-1\right)
				+\sum_i\log\frac{\sigma_\tau^2+\nu_l^2(\bx_i)}{\sigma_\tau^2+1}
				>-2\lambda_l\log \tM
				\right)>p^*.
			\end{split}
		\end{align}
		Define $s_{l,i}:=(1-\nu_l^2(\bx_i))/(\sigma_\tau^2+1)$.
		Then, the above display can be restated as
		\begin{align}\label{eq:m13}
			\begin{split}
				&\mathbb{P}_{\nu_l}\!\left(
				\sum_i\frac{y_i^2}{1+\sigma_\tau^2}
				\left(\frac{1}{1-s_{l,i}}-1\right)
				+\sum_i\log(1-s_{l,i})
				>-2\lambda_l\log \tM
				\right)>p^*.
			\end{split}
		\end{align}  
		By the mean value theorem, for some $\epsilon_{l,i},\epsilon_{l,i}'\in(0,s_{l,i})$, we have
		\[
		\tfrac{1}{1-s_{l,i}}=1+s_{l,i}+\tfrac{s_{l,i}^2}{(1-\epsilon_{l,i})^3},
		\qquad
		\log(1-s_{l,i})=-s_{l,i}-\tfrac{s_{l,i}^2}{2(1-\epsilon_{l,i}')^2}.
		\]
		For large $n$, $0<\epsilon_{l,i},\epsilon_{l,i}'<s_{l,i}\leq 2(1-\nu_l(\bx_i))<2\delta_n^{\beta^*} g_{\max}
		<\frac12$, where $g_{\max}=\max_{\bx} g(\bx)$. Since $\delta_n^{\beta^*}\to 0$ as $n\to\infty$, substituting and simplifying,
		\eqref{eq:m13} reduces to proving
		\begin{align}\label{eq:likelihood5}
			\mathbb{P}_{\nu_l}\!\left[
			\sum_i s_{l,i}\!\left(\frac{y_i^2}{1+\sigma_\tau^2}-1\right)
			+\sum_i s_{l,i}^2\!\left(\frac{y_i^2}{1+\sigma_\tau^2}-2\right)
			>-2\lambda_l\log \tM
			\right]>p^*.
		\end{align}
		Next, note that applying \prettyref{lem:lM} with $t=\log 3$ and $1-s_{l,i}=(\sigma_\tau^2+\nu_l^2(\bx_i))/(1+\sigma_\tau^2)$, we get
		\begin{align}\label{eq:quadratic_terms}
			\mathbb{P}_{\nu_l}\!\left(
			\sum_i\frac{y_i^2 s_{l,i}^2(1-s_{l,i})}{\sigma_\tau^2+\nu_l^2(\bx_i)}
			\leq\sum_i s_{l,i}^2(1-s_{l,i})
			-2\sqrt{\log 3}\sqrt{\sum_i s_{l,i}^4(1-s_{l,i})^2}
			\right)\leq\tfrac13.
		\end{align}
		Combining \eqref{eq:quadratic_terms} with \eqref{eq:likelihood5}, to establish \eqref{eq:m13} it
		suffices to prove, for some $p^*\geq0$,
		\begin{align}\label{eq:likelihood6}
			\begin{split}
				\mathbb{P}_{\nu_l}\!\left(
				\sum_i s_{l,i}\!\left(\frac{y_i^2}{1+\sigma_\tau^2}-1\right)
				-\sum_i s_{l,i}^2(1+s_{l,i})
				-2\sqrt{\log 3}\sqrt{\sum_i s_{l,i}^4}
				>-2\lambda_l\log \tM
				\right)>p^*+\tfrac13.
			\end{split}
		\end{align}
		
		Next, we bound $\sum_i s_{l,i}^r$ for $r=2,3,4$ via a Riemann sum approximation of $g_{n,j}(\cdot)$.
		Since the functions $\sth{g_{n,j}}_{j=1}^m$ have disjoint supports, for any $r\in\mathbb{Z}^+$, we have
		\[
		\sum_i\bigl(1-\nu_l(\bx_i)\bigr)^r
		=\sum_{\{j:(\bomega_l)_j=1\}}\sum_i g_{n,j}^r(\bx_i).
		\]
		For each fixed $j$, by the Riemann sum approximation, we get
		\begin{align}\label{eq:moments_f}
			\frac{1}{n}\sum_i g_{n,j}^r(\bx_i)
			= \delta_n^{r\beta^*}\,\frac{1}{n}\sum_i
			g^r\!\!\left(\tfrac{\bx_i-\ba_{n,j}}{\delta_n}\right)
			= \delta_n^{r\beta^*+d^*}\bigl(\|g\|_r^r+o(1)\bigr),
		\end{align}
		where the integral $\int_{[0,1]^{d^*}}g^r((\bx-\ba_{n,j})/\delta_n)\,d\bx
		=\delta_n^{d^*}\|g\|_r^r$ is approximated by the Riemann sum with error
		$o(\delta_n^{d^*})$ as $n\delta_n^{d^*}\to\infty$.  Summing over the $C_lm$
		nonzero entries of $\bomega_l$ (with $1\geq C_l\geq 1/16$) and using
		$s_{l,i}=(1-\nu_l(\bx_i))/(1+\sigma_\tau^2)$, we deduce
		\begin{align}\label{eq:sr_first}
			\frac{C_l m\,\delta_n^{r\beta^*+d^*}n\bigl(\|g\|_r^r+o(1)\bigr)}{(1+\sigma_\tau^2)^r}
			<\sum_i s_{l,i}^r<
			\frac{2^r C_l m\,\delta_n^{r\beta^*+d^*}n\bigl(\|g\|_r^r+o(1)\bigr)}{(1+\sigma_\tau^2)^r}.
		\end{align}
		
		For $\delta_n$ as in \eqref{eq:Mn}, we have
		$\frac14\leq\frac{\delta_n^{2\beta^*+d^*}n}{(1+\sigma_\tau^2)^2}\leq1$.
		Substituting the upper bounds for $\sum s_{l,i}^2,\sum s_{l,i}^3,\sum s_{l,i}^4$
		into \eqref{eq:likelihood6} and using $\delta_n=o(1)$, it suffices to
		prove
		\begin{align}\label{eq:likelihood8}
			\mathbb{P}_{\nu_l}\!\left(
			\sum_i s_{l,i}\!\left(\frac{y_i^2}{1+\sigma_\tau^2}-1\right)
			-4C_l m\,\|g\|_2^2+o(m)
			>-2\lambda_l\log \tM
			\right)>p^*+\tfrac13.
		\end{align}
		Next, we obtain a Gaussian approximation of the above probability via \prettyref{lem:BE}. To apply \prettyref{lem:BE}, we define $Z_i:=s_{l,i}(y_i^2/(1+\sigma_\tau^2)-1+s_{l,i})$.
		Since $\bx_i$'s are fixed, $y_i$ is a (non-identically distributed) random
		variable whose distribution depends on $\bx_i$ only through $\nu_l(\bx_i)$.
		Direct computation gives
		\begin{align}\label{eq:m14}
			\mathbb{E}(Z_i)=0,\quad
			\mathbb{E}(Z_i^2)=2s_{l,i}^2(1-s_{l,i})^2,\quad
			\mathbb{E}|Z_i|^3\leq 28\,s_{l,i}^3(1-s_{l,i})^3.
		\end{align}
		In view of the above deductions, we reorganize the terms in \eqref{eq:likelihood8} to get the equivalent statement to be shown
		\begin{align}\label{eq:likelihood9}
			\mathbb{P}_{\nu_l}\!\left(
			\sum_i {Z_i\over \sqrt{\sum_{i}\EE [Z_i^2]}}
			>{4C_l m\,\|g\|_2^2+ \sum_{i} s_{l,i}^2- o(m) -2\lambda_l\log \tM\over \sqrt{2\sum_i s_{l,i}^2(1-s_{l,i})^2}}
			\right)>p^*+\tfrac13.
		\end{align}
		Hence, using,
		\begin{align}\label{eq:bd:sum:sli}
			\sum_i s_{l,i}^2\leq 5C_l m\,\|g\|_2^2.
		\end{align}
		from \eqref{eq:sr_first} with $r=2$, $\delta_n^{2\beta^*+d^*}n=O(1)$, and $s_{l,i}=(1-\nu_l^2(\bx_i))/(\sigma_\tau^2+1)<1$, it suffices to show 
		\begin{align}\label{eq:likelihood10}
			\mathbb{P}_{\nu_l}\!\left(
			\sum_i {Z_i\over \sqrt{\sum_{i}\EE [Z_i^2]}}
			>\frac{m\sqrt{2}(9C_l\|g\|_2^2-\lambda_l/5)}{\sqrt{\sum_i s_{l,i}^2}}
			\right)>p^*+\tfrac13.
		\end{align}
		Now we are in a position to apply \prettyref{lem:BE}. As \eqref{eq:m14} implies 
		\[
		\frac{\mathbb{E}|Z_i|^3}{\mathbb{E}(Z_i^2)}
		\leq 14\,s_{l,i}(1-s_{l,i})
		\leq 28\bigl(1-\nu_l(\bx_i)\bigr)
		\leq 28\,\delta_n^{\beta^*} g_{\max},
		\]
		we use $\kappa=28\,\delta_n^{\beta^*} g_{\max}$, and let $F_{S_*}$ denote the CDF of $Z_i\over \sqrt{\EE[Z_i^2]}$ under $\nu_l$ to get from \prettyref{lem:BE}
		\begin{align}\label{eq:CLT:upper}
			\begin{split}
				\sup_x\bigl|F_{S_*}(x)-\Phi(x)\bigr|
				&= O\!\left(\frac{\delta_n^{\beta^*}}{\sqrt{\sum_i s_{l,i}^2}}\right)
				\overset{(a)}{=}O\!\left(\frac{\delta_n^{\beta^*}}{\sqrt{m}}\right)
				\overset{(b)}{=}O\!\left(\delta_n^{\beta^*+d^*/2}\right)=o(1),
			\end{split}
		\end{align}
		where (a) uses \eqref{eq:bd:sum:sli} and (b) uses $m=\delta_n^{-d^*}$. In view of the above, to prove \eqref{eq:likelihood8} it suffices to show
		\begin{align}\label{LR:afterBerryEsseen}
			\mathbb{P}\!\left(N(0,1)
			>\frac{m\sqrt{2}(9C_l\|g\|_2^2-\lambda_l/5)}{\sqrt{\sum_i s_{l,i}^2}}
			\right)>p^*+\tfrac13.
		\end{align}
		To show the above inequality, we first use \eqref{eq:bd:sum:sli}, to get
		\[
		\mathbb{P}\!\left(N(0,1)
		>\sqrt{10m}\!\left(9\|g\|_2\sqrt{C_l}
		-\frac{\lambda_l}{5\sqrt{C_l}\|g\|_2}\right)
		\right)>p^*+\tfrac13.
		\]
		Choose $\lambda_l=\sqrt{C_l/2}$ and $\|g\|_2<1/10$; then the argument of
		$N(0,1)$ is negative, and we may take $p^*=1/6$.
		This completes the proof.
	\end{proof}
	
	\section{Proof of technical results}
	
	\label{app:technical}

	\subsection{Proof of \prettyref{lmm:empirical-bound-1}}
	We use the Hanson-Wright inequality \prettyref{lmm:hanson wright} with the vector $\bxi=[\xi_1,\dots,\xi_n]$:
	\begin{align}
		\label{eq:m3}
		\mathbb{P}\!\left(
		\left|\bxi^\top D \bxi - \mathbb{E}[\bxi^\top D \bxi]\right| > t
		\right)
		\le
		\exp\left\{
		-c\min\!\left(
		\frac{t^2}{\|D\|_{HS}^2},
		\frac{t}{\|D\|_2}
		\right)
		\right\}.
	\end{align}
	We apply the above result with
	\[
	D = \Sigma^{*-1/2}(\tilde\Sigma-\Sigma^*)\Sigma^{*-1/2}.
	\]
	Then, we have
	\[
	\|D\|_{HS}^2
	= \mathrm{Tr}(D^2)
	= \mathrm{Tr}\!\left(
	\Sigma^{*-1}(\tilde\Sigma-\Sigma^*)\Sigma^{*-1}(\tilde\Sigma-\Sigma^*)
	\right),
	\]
	and
	\begin{align}
		\|D\|_2
		&\le
		\max_i f^*(\bx_i)^2
		\left|
		\frac1{\tilde f(\bx_i)^2}
		-
		\frac1{f^*(\bx_i)^2}
		\right|
		\\
		&\le
		\max_i f^*(\bx_i)^2
		\frac{\tilde f(\bx_i)^2 - f^*(\bx_i)^2}{\tilde f(\bx_i)^2 f^*(\bx_i)^2}
		\le
		\|\tilde f - f^*\|_\infty
		\max_i \frac{\tilde f(\bx_i)+f^*(\bx_i)}{(\tilde f(\bx_i))^2}
		\leq 
		\|\tilde f - f^*\|_\infty
		\frac{f_{\max}}{(f_{\min})^2}.
		\nonumber
	\end{align}
	Combining the above results, we continue \eqref{eq:m3} to get the desired upper bound.

	\subsection{Proof of \prettyref{lmm:ineq1}}
		Let $g(x)=x-\log x-a(x-1)^2$. Then we have
		\[
		g'(x)=1-\frac1x-2a(x-1)=-{2a(x-1)(x-\frac 1{2a})\over x}.
		\]
		Setting $g'(x)=0$ gives us two solutions $
		x=1, x=\frac1{2a}.$
		Next we note that $g(1)=0$, and
		\begin{itemize}
			\item $0<x<1$ implies $g'(x)<0$
			\item $1<x\leq \frac 1{2a}$ implies $g'(x)>0$. 
		\end{itemize}
		Thus
		$
		g(x)\ge0$ for $0 < x\le\frac1{2a}$, as desired.

	\subsection{Proof of \prettyref{lmm:chaining}}

	Given $\gamma\geq \delta$, consider the $\delta$-covering set given by $\calN(\delta,\calG_{m,s},\norm{\cdot}_\infty)$. Consider the set of functions $\calS = \{f^\sparse\in \mathcal{G}_{m,s}, \| f^\sparse- \bar f^\sparse\|_n \le \gamma\}$ and the subset of the covering set
	\begin{align}
		\bar \calN_\delta = \calN(\delta,\calG_{m,s},\norm{\cdot}_\infty)
		\cap
		\{f^\sparse\in \mathcal{G}_{m,s}, \| f^\sparse- f^*\|_n \le \gamma\}.
	\end{align}
	Given any $f^\sparse\in \calS$, consider $\tilde f^\sparse \in \bar \calN_{\delta}$ such that 
	\begin{align}
		\label{eq:m25}
		\|f^\sparse-\tilde f^\sparse\|_\infty\leq \delta.	
	\end{align}
	Then from \eqref{eq:ineq-estimation-highdim-2} we get
	\begin{align}
		\label{eq:m24}
		\frac 1n\abs{\by^\top(\Sigmastar-\Sigmatilde^\sparse)\by
			-
			\mathbb{E}\!\left[
			\by^\top(\Sigmastar-\Sigmatilde^\sparse)\by
			\right]}
		\leq \tilde C_1
		\sth{ \epsilon\gamma^2
			+\pth{1+\frac 1{2\epsilon}}{\pth{\nu_{n}+s\varrho_{n} + \frac tn}}}
	\end{align}
	with probability $1-e^{-t}$, where we used that $\varrho_{n,\delta}\leq C_{11}\varrho_n,\nu_{n,\delta}\leq C_{12}\nu_n$ for $\delta$ as in \eqref{eq:delta-sparse}, and $\tilde C_1,C_{11},C_{12}>0$ are constants. Using \prettyref{lmm:bernstein}, we get $\frac 1n \sum_{i=1}^n y_i^2\leq C_{21}+C_{22}\sqrt{\frac tn}+C_{23}\frac tn$ with probability $1-e^{-t}$, where $C_{21},C_{22},C_{23}>0$ are constants depending on $f_{\max},f_{\min}$. Hence, for a constant $C_{31}=C_{31}(f_{\max},f_{\min})$
	\begin{align}
		&\frac 1n\abs{\by^\top(\Sigma^\sparse-\Sigmatilde^\sparse)\by
			-
			\mathbb{E}\!\left[
			\by^\top(\Sigma^\sparse-\Sigmatilde^\sparse)\by
			\right]}
		\notag\\
		&\leq 
		\sup_{\bx}|(f^\sparse(\bx))^{-2}-(\tilde f^\sparse(\bx))^{-2}|
		\pth{\frac 1n \sum_{i=1}^n y_i^2 + \EE\qth{\frac 1n \sum_{i=1}^n y_i^2}}
		\notag\\
		&
		\stepa{\leq}
		C_{31}
		\|f^\sparse-\tilde f^\sparse\|_\infty
		\pth{1+\sqrt{\frac tn}+\frac tn}
		\leq 
		C_{32}
		\delta
		\pth{1+\frac tn}
		\stepb{\leq} C_{33}(\nu_n+\frac tn),		
	\end{align}
	where (a) follows by using $a\leq \max\{1,a^2\}\leq 1+a^2$ with $a=\frac tn$ and(b) follows by using $\delta\leq \min\{\nu_n,1\}$. In view of the above display, we use \eqref{eq:m25} to get for some constant $C_3>0$
	\begin{align*}
		\label{eq:m26}
		&\frac 1n\abs{\by^\top(\Sigma^*-\Sigma^\sparse)\by
			-
			\mathbb{E}\!\left[
			\by^\top(\Sigma^*-\Sigma^\sparse)\by
			\right]}
		\notag\\
		&
		\leq \frac 1n\abs{\by^\top(\Sigma^*-\Sigmatilde^\sparse)\by
			-
			\mathbb{E}\!\left[
			\by^\top(\Sigma^*-\Sigmatilde^\sparse)\by
			\right]}
		+\frac 1n\abs{\by^\top(\Sigma^\sparse-\Sigmatilde^\sparse)\by
			-
			\mathbb{E}\!\left[
			\by^\top(\Sigma^\sparse-\Sigmatilde^\sparse)\by
			\right]}
		\notag\\
		&\leq 
		C_3
		\sth{ \epsilon\gamma^2
			+\pth{1+\frac 1{2\epsilon}}{\pth{\nu_{n}+s\varrho_{n} + \frac tn}}},
	\end{align*}
	with probability $1-2e^{-t}$. This completes the proof.

\end{document}